\documentclass[leqno,12pt]{article}
\def\today{${\scriptscriptstyle\number\day-\number\month-\number\year}$}
\usepackage{graphicx}
\usepackage{fancyhdr}
\usepackage{amsmath}
\usepackage{amsthm}
\usepackage{amsfonts}
\newtheorem{theorem}{Theorem}[section]

\newtheorem{lemma}[theorem]{Lemma}
\newtheorem{remark}[theorem]{Remark}

\def\address#1{{\center{#1}}}
\date{}

\def\m@th{\mathsurround=0pt}
\def\eqal#1{\null\,\vcenter{\openup\jot\m@th
 \ialign{\strut\hfil$\displaystyle{##}$&&$\displaystyle{{}##}$\hfil
 \crcr#1\crcr}}\,}
\def\matrix#1{\null\,\vcenter{\normalbaselines\m@th
 \ialign{\hfil$##$\hfil&&\quad\hfil$##$\hfil\crcr
 \mathstrut\crcr\noalign{\kern-\baselineskip}
 #1\crcr\mathstrut\crcr\noalign{\kern-\baselineskip}}}\,}

\def\N{\mathbb{N}}
\def\R{\mathbb{R}}

\def\supp{{\rm supp}}
\def\divv{{\rm div\,}}

\title{Global regular motions for compressible barotropic viscous fluids.
Stability}
\author{H-O. Bae$^1$ and Wojciech M. Zaj\c aczkowski$^2$}

\begin{document}
\input amssym.def
\input amssym.tex
\maketitle
\thispagestyle{fancy}
\address{$^1$  Department of Mathematics, Ajou University, Suwon, South Korea\\
$^2$ Institute of Mathematics, Polish Academy of Sciences,\\
\'Sniadeckich 8, 00-656 Warsaw, Poland,\\
Institute of Mathematics and Cryptology, Cybernetics Faculty,\\
Military University of Technology,\\
Kaliskiego 2, 00-908 Warsaw, Poland\\
e-mail:wz@impan.gov.pl\\}

\begin{abstract}
We consider viscous compressible barotropic motions in a~boun\-ded domain $\Omega\subset\R^3$ with the Dirichlet boundary conditions for velocity. We assume the existence of some special sufficiently regular solutions $v_s$ (velocity), $\varrho_s$ (density) of the problem. By the special solutions we can choose spherically symmetric solutions. Let $v$, $\varrho$ be a~solution to our problem. Then we are looking for differences $u=v-v_s$,
$\eta=\varrho-\varrho_s$. We prove existence of $u$, $\eta$ such that $u,\eta\in L_\infty(kT,(k+1)T;H^2(\Omega))$, $u_t,\eta_t\in L_\infty(kT,(k+1)T;H^1(\Omega))$, $u\in L_2(kT,(k+1)T;H^3(\Omega))$,
$u_t\in L_2(kT,(k+1)T;H^2(\Omega))$, where $T>0$ is fixed and $k\in\N\cup\{0\}$. Moreover, $u$, $\eta$ are sufficiently small in the above norms. This also means that stability of the special solutions $v_s$, $\varrho_s$ is proved. Finally, we proved existence of solutions such that $v=v_s+u$, $\varrho=\varrho_s+\eta$.

\noindent
{\bf Mathematical Subject Classification (2010):} 35A01, 35Q30, 76N10

\noindent
{\bf Key words:} compressible viscous barotropic fluids, Dirichlet boundary conditions, global existence of regular solutions, stability of sphe\-rically symmetric solutions
\end{abstract}

\section{Introduction}\label{s1}
\setcounter{equation}{0}

We consider the motions of viscous compressible barotropic fluids in\break a~bounded
domain $\Omega\subset\R^3$ described by the following problem
\begin{equation}
\eqal{
&\varrho(v_t+v\cdot\nabla v)-\mu\Delta v-\nu\nabla\divv v+\nabla p=\varrho f
\quad &{\rm in}\ \ \Omega^T\equiv\Omega\times(0,T),\cr
&\varrho_t+\divv(\varrho v)=0\quad &{\rm in}\ \ \Omega^T,\cr
&\varrho|_{t=0}=\varrho_0,\ \ v|_{t=0}=v_0\quad &{\rm in}\ \ \Omega,\cr
&v=0\quad &{\rm on}\ \ S^T=S\times(0,T),\cr}
\label{1.1}
\end{equation}
where $S$ is the boundary of $\Omega$. By $x=(x_1,x_2,x_3)$ we denote the
Cartesian coordinates, \hskip-1pt $\varrho=\varrho(x,t)\in\R_+$ \hskip-1pt is \hskip-1pt the \hskip-1pt density \hskip-1pt of \hskip-1pt the \hskip-1pt fluid, \hskip-1pt $v=(v_1(x,t),v_2(x,t),\break v_3(x,t))\in\R^3$ is the velocity,
$p=p(\varrho)=A\varrho^\varkappa$, $\varkappa>1$, $A$ positive constant, the
pressure, $f=(f_1(x,t),f_2(x,t),f_3(x,t))\in\R^3$ the external force field. By
$\mu$, $\nu$ we denote positive viscosity coefficients satisfying the
following thermodynamic restrictions
\begin{equation}
\mu>0,\quad \nu>0.
\label{1.2}
\end{equation}
Finally, by the dot we denote the scalar product in $\R^3$.

\noindent
Since we are interested to prove global regular solutions to (\ref{1.1}) and
since such result can not be expected in the nearest future (if at all) we
are looking for such solutions to (\ref{1.1}) which are close to some known
special regular global solutions to (\ref{1.1}) (for example: spherically
symmetric, two-dimensional).

\noindent
Global existence of regular spherically symmetric solutions to problem
(\ref{1.1}) in a~domain between two spheres is proved in \cite{CK}. Global
existence of regular two-dimensional solutions is proved in \cite{KW} for
equations $(\ref{1.1})_{1,2}$, where $\nu=\nu(\varrho)=\varrho^\gamma$,
$\gamma>3$, and for the slip boundary conditions. Therefore, we restrict our
considerations to show stability of spherically symmetric solutions.
Moreover, stability of two-dimensional solutions must be performed in
a different way.

\noindent
Hence, we assume that $\varrho_s,v_s,p_s=A\varrho_s^\varkappa$ are the special
regular solutions satisfying the problem
\begin{equation}
\eqal{
&\varrho_s(v_{st}+v_s\cdot\nabla v_s)-\mu\Delta v_s-\nu\nabla\divv v_s+
\nabla p_s=\varrho_sf_s\quad &{\rm in}\ \ \Omega^T,\cr
&\varrho_{st}+\divv(\varrho_sv_s)=0\quad &{\rm in}\ \ \Omega^T,\cr
&\varrho_s|_{t=0}=\varrho_{s0},\ \ v_s|_{t=0}=v_{s0}\quad &{\rm in}\ \ \Omega,\cr
&v_s|_S=0\quad &{\rm on}\ \ S^T.\cr}
\label{1.3}
\end{equation}
Then we are looking for solutions to problem (\ref{1.1}) in the form
\begin{equation}
v=v_s+u,\quad \varrho=\varrho_s+\eta,\quad p=p_s+q,\quad f=f_s+g,
\label{1.4}
\end{equation}
where $u$, $\eta$, $q$ are solutions to the problem
\begin{equation}
\eqal{
&\varrho_s(u_t+v_s\cdot\nabla_u)-\mu\Delta u-\nu\nabla\divv u+\nabla q\cr
&=\eta f_s+\varrho_sg-[\eta(v_{st}+(v_s+u)\cdot\nabla v_s)+\varrho_su\cdot
\nabla v_s]\quad {\rm in}\ \ \Omega^T,\cr
&\quad-\eta u_t-[\eta(u+v_s)+\varrho_su]\nabla u\equiv\bar f\cr}
\label{1.5}
\end{equation}
\begin{equation}
\eta_t+v_s\cdot\nabla\eta+\varrho_s\divv u=-\eta\divv v_s-u\cdot\nabla\varrho_s-
u\cdot\nabla\varrho_s-\divv(\eta u)\equiv\bar h\quad {\rm in}\ \ \Omega^T,
\label{1.6}
\end{equation}
\begin{equation}
\eqal{
&u|_{t=0}=u_0\equiv v_0-v_{s0},\ \ \eta|_{t=0}=\eta_0\equiv
\varrho_0-\varrho_{s0}\quad &{\rm in}\ \ \Omega,\cr
&u|_S=0\quad &{\rm on}\ \ S^T.\cr}
\label{1.7}
\end{equation}
Equation (\ref{1.6}) is not compatible with the main operator in (\ref{1.5}),
where $\nabla q$ appears. Therefore, we replace (\ref{1.6}) by an equation
for $q$.

\noindent
Since $p=A\varrho^\varkappa$, equation $(\ref{1.1})_2$ yields
\begin{equation}
p_t+v\cdot\nabla p+\varkappa p\divv v=0.
\label{1.8}
\end{equation}
Similarly, for $p_s=A\varrho_s^\varkappa$ equation $(\ref{1.3})_2$ implies
\begin{equation}
p_{st}+v_s\cdot\nabla p_s+\varkappa p_s\divv v_s=0.
\label{1.9}
\end{equation}
Using definition of $q$ from (\ref{1.4}), equations (\ref{1.8}) and
(\ref{1.9}) imply
\begin{equation}
{1\over\varkappa p_s}(q_t+v_s\cdot\nabla q)+\divv u=-{u\over\varkappa p_s}
\cdot\nabla(p_s+q)-{q\over p_s}\divv(v_s+u)\equiv\bar h.
\label{1.10}
\end{equation}
For equations (\ref{1.5}) and (\ref{1.10}) we have the following initial and
boundary conditions
\begin{equation}
u|_{t=0}=u_0,\ \ q|_{t=0}=A(\varrho_0^\varkappa-\varrho_{s0}^\varkappa)\equiv
q_0,
\label{1.11}
\end{equation}
\begin{equation}
u|_S=0,\ \ v_s|_S=0.
\label{1.12}
\end{equation}
Our aim is to prove global existence of regular solutions to problem
(\ref{1.5}), (\ref{1.10})--(\ref{1.12}) under assumption that $u_0$, $q_0$ are
sufficiently small in appropriate norms and $(\varrho_s,v_s)$ is
sufficiently regular solution to problem (\ref{1.3}). Since in equations
(\ref{1.5}) and (\ref{1.10}) $\eta$ and $q$ appear together, we have to find
relations between them. In view of formulas $q=p(\varrho)-p(\varrho_s)=
A(\varrho^\varkappa-\varrho_s^\varkappa)$ and $\eta=\varrho-\varrho_s$
we have the relations
\begin{equation}
q=\varkappa A\tilde\varrho^{\varkappa-1}\eta\quad {\rm and}\quad
\eta={1\over\varkappa A}\tilde q^{1/\varkappa}q,
\label{1.13}
\end{equation}
where $\tilde\varrho\in(\varrho,\varrho_s)$ and $\tilde q\in(p,p_s)$.
To justify these formulas we need existence of positive constants $\varrho_*$
and $\varrho^*$ such that
\begin{equation}
\varrho_*\le\varrho_s\le\varrho^*\quad {\rm and}\quad
{1\over2}\varrho_*\le\varrho\le2\varrho ^*.
\label{1.14}
\end{equation}
\noindent
The aim of this paper is to prove existence of global regular solutions to
problem (\ref{1.5}), (\ref{1.10})--(\ref{1.12}) under assumptions that
$u_0$, $\eta_0$, $g$ are sufficiently small in corresponding norms and
$(\varrho_s,v_s)$ is a~sufficiently regular solution to (\ref{1.2}). Hence, we
also prove existence of global regular solutions to (\ref{1.1}) which
remain sufficiently close to special solutions to (\ref{1.2}) for all time.
In Section \ref{s3} we prove existence of local regular solutions to
(\ref{1.5}), (\ref{1.10})--(\ref{1.12}) by the method of successive
approximations. Moreover, the existence is proved in such form that the
existence time is inversely proportional to the corresponding norms of the
initial data and the external force $g$. We derived in Section \ref{s4}
differential inequality (\ref{4.82}) which makes possible an extension of the
local solution step by step in time. Notation is introduced in Section \ref{s2}.

\noindent
Now we formulate the main results of this paper.

\noindent
Let $\Gamma_{k-1}^k(\Omega)=\{u:\ u\in H^k(\Omega),u_t\in H^{k-1}(\Omega)\}$,
$k\in\N$.

\begin{remark}\label{r1.1} (see \cite{CK})
Assume that there exists a~global special solution $(\varrho_s,v_s)$ to problem
(\ref{1.3}) belonging to space $\bar X(\Omega\times(kT,(k+1)T))$, $k\in\N_0$
(see definition below). Moreover, the solution is such that there exists
positive constants $\varrho_*<\varrho^*$ and
$\varrho_*\le\varrho_s(x,t)\le\varrho^*$.
\end{remark}

$$\eqal{
&\bar X(\Omega\times(kT,(k+1)T))=\{(\varrho,v):\ \cr
&\quad\|\varrho,v\|_{\bar X(\Omega\times(kT,(k+1)T))}=
\|\varrho,v\|_{L_\infty(kT,(k+1)T;\Gamma_1^2(\Omega))}\cr
&+\quad\|v_{,tt}\|_{L_\infty(kT,(k+1)T;L_2(\Omega))}<\infty\},\cr}
$$
where $\|\varrho,v\|_H=\|\varrho\|_H+\|v\|_H$ and $H$ can be any space used in this paper.

\noindent
Let
$$\eqal{
&\varphi_1(t)=\|u(t),q(t)\|_{\Gamma_1^2(\Omega)}^2,\cr
&\Phi_1=\|u(t)\|_{\Gamma_2^3(\Omega)}^2+\|q(t)\|_{\Gamma_1^2(\Omega)}^2,\cr
&(u,q)\in{\frak N}(\Omega\times(kT,(k+1)T))\quad {\rm if}\quad \sup_{kT\le t\le(k+1)T}
\varphi_1(t)<\infty,\cr
&(u,q)\in{\frak M}(\Omega\times(kT,(k+1)T))\quad {\rm if}\quad \sup_{kT\le t\le(k+1)T}
\varphi_1(t)\,+\!\!\!\!\intop_{kT}^{(k+1)T}\!\Phi_1(t)dt<\infty.\cr}
$$

\begin{theorem}\label{t1.2} [Global existence of solutions to (\ref{1.5}),
(\ref{1.10})--(\ref{1.12})]
Let $k\in\N_0$ and $T>0$ be given. Assume that $\varrho_s,v_s\in\bar X(\Omega\times(kT,(k+1)T))$,
$f_s,g\in L_\infty(kT,(k+1)T;\Gamma_0^1(\Omega))$. Assume also that there exists a~constant $\gamma$ sufficiently  small and
$\varphi_1(0)\le\gamma,\sup_{k\in\N_0}$, $\|g\|_{L_\infty(kT,(k+1)T;\Gamma_0^1(\Omega))}\le c\gamma$.\\
Then there exists a~global solution to problem (\ref{1.5}),
(\ref{1.10})--(\ref{1.2}) such that $(u,q)\in{\frak M}(kT,(k+1)T;\Omega)$ for any $k\in\N_0$ and $\varphi_1(t)\le\gamma$, for any $t\in\R_+$.
\end{theorem}

\begin{theorem}\label{t1.3}[Global existence of solutions to problem
(\ref{1.1})].
Let there exists a~special solution to (\ref{1.3}) described by
Remark \ref{r1.1}.
Let the assumptions of Theorem \ref{t1.2} hold. Then there exists a~global
regular solution to problem (\ref{1.1}) in the form
$$
v=v_s+u,\quad \varrho=\varrho_s+\eta
$$
such that $(u,\eta)\in{\frak M}(\Omega\times(kT,(k+1)T))$, $k\in\N_0$ and
$(v_s,\varrho_s)\in\bar X(\Omega\times(kT,(k+1)T))$, $k\in\N_0$.\\
Moreover the following estimate holds
$$
\sup_k\|v,\varrho\|_{{\frak M}(\Omega\times(kT,(k+1)T))}\le c D\gamma+\sup_k
\|v_s,\varrho_s\|_{\bar X(\Omega\times(kT,(k+1)T))}
$$
where $D$ is a~function of $\sup_k\|\varrho_s,v_s\|_{\bar X(\Omega\times[kT,(k+1)T])}$\\
and $\sup_k\|f_s\|_{L_\infty(kT,(k+1);\Gamma_0^1(\Omega))}$ (see Remark \ref{r4.8}).
\end{theorem}

\section{Notation and auxiliary results}\label{s2}
\setcounter{equation}{0}

By $\|\cdot\|_{l,\Omega}$, $l\ge0$ and $|\cdot|_{p,\Omega}$,
$1\le p\le\infty$, we denote the norms of the usual Sobolev spaces
$W_2^l(\Omega)=H^l(\Omega)$ and $L_p(\Omega)$ spaces, respectively.

\noindent
Next, we introduce the space $\Gamma_k^l(\Omega)$ of functions $u$ with the
finite norm
$$
\|u\|_{\Gamma_k^l(\Omega)}=\sum_{i\le l-k}\|\partial_t^iu\|_{l-i,\Omega}
\equiv|u|_{l,k,\Omega},
$$
where $l>0$ and $k\ge0$, $k\le l$. In this paper we use that
$l,k\in\N_0=\N\cup\{0\}$.

\noindent
Let $u_1,u_2,\dots,u_n$ be given functions. Then
$$
\|u_1,u_2,\dots,u_n\|_X^2=\sum_{i=1}^n\|u_i\|_X^2,
$$
where $X$ is any used in this paper space.

\noindent
For distinguish time dependence of functions we introduce the notation
$$\eqal{
&\|u\|_{L_p(0,T;H^k(\Omega))}=\|u\|_{k,p,\Omega^T},\cr
&\|u\|_{L_q(0,T;L_p(\Omega))}=\langle u\rangle_{p,q,\Omega^T},\cr}
$$
where $p,q\in[1,\infty]$. Moreover,
$$
\|u\|_{L_p(0,T;\Gamma_k^l(\Omega))}=|u|_{l,k,p,\Omega^T}.
$$
The index $\Omega$ is always omitted. If $\Omega$ is replaced by a~subdomain
of $\Omega$, we left it in the above notation. Finally, we present the compatibility conditions
$$
\|u\|_{0,\Omega}=|u|_{2,\Omega},\quad |u|_{l,k,\Omega}=\|u\|_{l,\Omega},\quad
|u|_{l,l,p,\Omega^T}=\|u\|_{l,p,\Omega^T}.
$$
By $c$ we denote the generic constant which changes its value from formula to
formula. By $\varphi$ we denote the generic function which is always positive
increasing function of its arguments and it changes its form from formula
to formula.
Let $w$ be any function in this paper. Then we denote $\partial_tw=w_{,t}$.
For functions $w_n$, $w_s$ we have $\partial_tw_n=w_{n,t}$,
$\partial_tw_s=w_{s,t}$ and so on. Let $\alpha(u)=t^{1/2}\|u\|_{3,2,\Omega^t}$.

\noindent
It is convenient to introduce the elliptic operator
\begin{equation}
Au=\mu\Delta u+\nu\nabla\divv u.
\label{2.1}
\end{equation}
From (\ref{1.9}) we obtain the equation of continuity
\begin{equation}
\bigg({1\over p_s}\bigg)_{,t}+\divv\bigg({v_s\over p_s}\bigg)=
{\varkappa+1\over p_s}\divv v_s.
\label{2.2}
\end{equation}
From continuity equations $(\ref{1.1})_2$ and $(\ref{1.3})_2$ we have
${d\over dt}\intop_\Omega\varrho dx=0$, \break
${d\over dt}\intop_\Omega\varrho_sdx=0$ so
\begin{equation}
\intop_\Omega\eta dx=\intop_\Omega(\varrho_0-\varrho_{s0})dx\quad
{\rm for\ any}\ \ t\in\R_+.
\label{2.3}
\end{equation}
Inequalities (\ref{1.13}) and (\ref{1.14}) imply
\begin{equation}
|\eta|_{p,\Omega}\le c_1(\varrho_*,\varrho^*)|q|_{p,\Omega},\quad
|q|_{p,\Omega}\le c_2(\varrho_*,\varrho^*)|\eta|_{p,\Omega}.
\label{2.4}
\end{equation}
To obtain an estimate for spatial derivatives higher than the first we need
local considerations. For this we introduce a~partition of unity
$(\{\tilde\Omega_i\},\{\zeta_i\})$, where $\Omega=\bigcup_i\tilde\Omega_i$
and $\tilde\Omega_i=\supp\zeta_i$. Let $\tilde\Omega$ be one of the
$\tilde\Omega'_is$ and $\zeta(x)=\zeta_i(x)$ the corresponding function.
If $\tilde\Omega$ is an interior subdomain, then let $\tilde\omega$ be such
that $\bar{\tilde\omega}\subset\tilde\Omega$ and $\zeta(x)=1$ for $x\in\tilde\omega$. Otherwise we assume that
$\bar{\tilde\Omega}\cap S\not=\phi$, $\bar{\tilde\omega}\cap S\not=\phi$,
$\bar{\tilde\omega}\subset\bar{\tilde\Omega}$. Let $\xi\in\bar{\tilde\omega}\cap S\subset\bar{\tilde\Omega}\cap S\equiv\tilde S$.
Then we introduce new coordinates $y=Y(x)$ with origin at $\xi$. The mapping
$Y$ is a~composition of translation and rotation. Assume that $\tilde S$ in
the coordinates $y$ is described by $y_3=F(y_1,y_2)$, where $F$ is sufficiently
regular. Then
$$\eqal{
&\tilde\Omega=\{y:\ |y_i|<2\lambda,i=1,2,F(y')<y_3<F(y')+2\lambda,
y'=(y_1,y_2)\},\cr
&\tilde\omega=\{y:\ |y_i|<\lambda,i=1,2,F(y')<y_3<F(y')+\lambda,
y'=(y_1,y_2)\}.\cr}
$$
Further, we introduce new variables by
$$
z_i=y_i,\quad i=1,2,\quad z_3=y_3-\tilde F(y'),\quad y\in\tilde\Omega,
$$
which will be denoted by $z=\Phi(y)$, where $\tilde F$ is an extension of $F$
to $\tilde\Omega$. Let
$$
\hat\Omega=\Phi(\tilde\Omega)=\{z:\ |z_i|<2\lambda,i=1,2,0<z_3<2\lambda\}\quad
{\rm and}\quad \hat S=\Phi(\tilde S).
$$
Hence $\hat S=\{z:\ z_3=0,|z_i|<2\lambda,i=1,2\}$.

\noindent
Let $\Psi=\Phi\circ Y$. Then we introduce the notation
$\hat\nabla_k={\partial z_i\over\partial x_k}\nabla_{z_i}|_{x=\Psi^{-1}(z)}$.
Moreover, for interior subdomains we denote $\tilde u(x)=u(x)\zeta(x)$,
$x\in\tilde\Omega$, $\tilde\Omega\cap S=\phi$, and for boundary subdomains we
have $\tilde u(z)=\hat u(z)\hat\zeta(z)$, $z\in\hat\Omega=\Psi(\tilde\Omega)$,
$\tilde\Omega\cap S\not=\phi$, where $\hat u(z)=u(\Psi^{-1}(z))$.

\noindent
We use such notation that $\tau$ replaces $z'=(z_1,z_2)$ and $n=z_3$. Now we
transform equations (\ref{1.5}) and (\ref{1.10}) to the local coordinates
(we restrict the considerations to neighborhoods near the boundary only)
\begin{equation}
\hat\varrho_s\tilde u_t+\hat\varrho_s\hat v_s\cdot\hat\nabla\tilde u-
\mu\Delta_z\tilde u-\nu\nabla_z\divv_z\tilde u+\nabla_z\tilde q=
\tilde{\bar f}+k_1,
\label{2.5}
\end{equation}
\begin{equation}
{1\over\varkappa\hat p_s}(\tilde q_t+\hat v_s\cdot\hat\nabla\tilde q)+
\divv_z\tilde u=\tilde{\bar h}+k_2,
\label{2.6}
\end{equation}
where $\nabla_z$, $\Delta_z$, $\divv_z$ means the usual operators with respect to variables $z$, $k_i=k_{i1}+k_{i2}$, $i=1,2$, and
\begin{equation}
k_{11}=(\nabla_z-\hat\nabla)\tilde q-\mu(\Delta_z-\hat\Delta)\tilde u-\nu
(\nabla_z\divv_z-\hat\nabla\hat\divv)\tilde u
\label{2.7}
\end{equation}
\begin{equation}
\eqal{
k_{12}&=\hat q\hat\nabla\hat\zeta-\mu(2\hat\nabla_k\hat u\hat\nabla_k\hat\zeta+
\hat u\hat\Delta\hat\zeta)-\nu(\hat\divv\hat u\hat\nabla\hat\zeta\cr
&\quad+\hat\nabla\hat u_k\hat\nabla_k\hat\zeta+\hat u_k\hat\nabla\hat\nabla_k
\hat\zeta)\cr}
\label{2.8}
\end{equation}
\begin{equation}
k_{21}=(\divv_z-\hat\divv)\tilde u,
\label{2.9}
\end{equation}
\begin{equation}
k_{22}=(\hat q/\varkappa\hat p_s)\hat v_s\cdot\hat\nabla\hat\zeta+
\hat u\cdot\hat\nabla\hat\zeta,
\label{2.10}
\end{equation}
where the summation convention over repeated indices is assumed. Finally
\begin{equation}
\eqal{\tilde{\bar f}&=\hat\eta\tilde f_s+\hat\varrho_s\tilde g-
[\hat\eta(\hat v_{st}+(\hat v_s+\hat u)\cdot\hat\nabla\hat v_s)+\hat\varrho_s\hat u
\cdot\hat\nabla\hat v_s]\hat\zeta\cr
&\quad-\hat\eta\tilde u_t-[\hat\eta(\hat u+\hat v_s)+\hat\varrho_s\hat u]\cdot
\hat\nabla\hat u\hat\zeta\cr}
\label{2.11}
\end{equation}
and
\begin{equation}
\tilde{\bar h}=-{\hat u\over\varkappa\hat p_s}\cdot\hat\nabla(\hat p_s+\hat q)
\hat\zeta-{\tilde q\over\hat p_s}\hat\divv(\hat v_s+\hat u).
\label{2.12}
\end{equation}
Applying operator $(\mu+\nu)\nabla_z$ to (\ref{2.6}) and adding the result to
(\ref{2.5}) one has
\begin{equation}
\eqal{
&{(\mu+\nu)\over\varkappa}\nabla_z{1\over\hat p_s}\tilde q_t+\nabla_z\tilde q=
\mu(\Delta_z\tilde u-\nabla_z\divv_z\tilde u)-\hat\varrho_s\tilde u_t-
\hat\varrho_s\hat v_s\cdot\hat\nabla\tilde u\cr
&\quad+\tilde{\bar f}+k_1-{\mu+\nu\over\varkappa}\nabla_z\bigg({1\over\hat p_s}
\hat v_s\cdot\hat\nabla\tilde q\bigg)+(\mu+\nu)(\nabla_z\tilde{\bar h}+
\nabla_zk_2).\cr}
\label{2.13}
\end{equation}
Moreover, we express (\ref{2.5}) in the form
\begin{equation}
\eqal{
&(\mu+\nu)\nabla_z\divv_z\tilde u=-\mu(\Delta_z\tilde u-
\nabla_z\divv_z\tilde u)+\hat\varrho_s\tilde u_t+\hat\varrho_s\hat v_s\cdot
\nabla\tilde u\cr
&\quad+\nabla_z\tilde q-\tilde{\bar f}-k_1.\cr}
\label{2.14}
\end{equation}
Finally, we introduce the notation used in this paper
\begin{equation}
\eqal{
&\varphi_1=|u(t)|_{2,1}^2+|q(t)|_{2,1}^2,\quad
\tilde\varphi_1=|\tilde u(t)|_{2,1}^2+|\tilde q(t)|_{2,1}^2,\cr
&\hat\varphi_1=|\hat u(t)|_{2,1,\hat\Omega}^2+|\hat q(t)|_{2,1,\hat\Omega}^2,\cr
&\Phi_1=|u(t)|_{3,2}^2,\quad \tilde\Phi_1=|\tilde u(t)|_{3,2}^2,\quad
\hat\Phi_1=|\hat u(t)|_{3,2,\hat\Omega}^2,\cr
&A_1=|v_s(t)|_{2,1}^2+|\varrho_s(t)|_{2,1}^2,\quad
\tilde A_1=|\tilde v_s(t)|_{2,1}^2+|\tilde\varrho_s(t)|_{2,1}^2,\cr
&\hat A_1=|\hat v_s(t)|_{2,1,\hat\Omega}^2+
|\hat\varrho_s(t)|_{2,1,\hat\Omega}^2,\cr
&A_2=\|v_s(t)\|_3^2+\|\varrho_s(t)\|_3^2,\quad
\tilde A_2=\|\tilde v_s(t)\|_3^2+\|\tilde\varrho_s(t)\|_3^2,\cr
&\hat A_2=\|\hat v_s(t)\|_{3,\hat\Omega}^2+
\|\hat\varrho_s(t)\|_{3,\hat\Omega}^2,\cr
&A_3=\|v_{st}(t)\|_2^2+\|\varrho_{st}(t)\|_2^2,\cr
&A_4=A_2+A_3+\|v_{stt}\|_0^2\cr}
\label{2.15}
\end{equation}

\begin{remark}\label{r2.1}
Let $A$ be the elliptic operator
$$
Au=\mu\Delta u+\nu\nabla\divv u.
$$
For $u,v\in H^2(\Omega)\cap H_0^1(\Omega)$ we have the integration by parts
formula
$$
-\intop_\Omega v\cdot Audx=\intop_\Omega(\mu\nabla v\cdot\nabla u+\nu\divv v
\divv u)dx\equiv\intop_\Omega A^{1/2}v\cdot A^{1/2}udx,
$$
where
$$
A^{1/2}u=(\sqrt{\mu}\nabla u,\sqrt{\nu}\divv u).
$$
\end{remark}

\noindent
From \cite[Sect. 15]{BIN} we have the interpolation
\begin{equation}
|u|_3\le\varepsilon^{1/2}|u_{,x}|_2+c\varepsilon^{-1/2}|u|_2,
\label{2.16}
\end{equation}
where $\varepsilon\in(0,1)$. We frequently use the interpolation
\begin{equation}
\|u\|_s\le\varepsilon\|u\|_{s+k}+c/\varepsilon\|u\|_0,\quad
s,k\in\N,\ \ s+k\le3,\ \ \varepsilon\in(0,1).
\label{2.17}
\end{equation}

\section{Local existence}\label{s3}
\setcounter{equation}{0}

To prove local existence of solutions to problem (\ref{1.5}),
(\ref{1.10})--(\ref{1.12}) we formulate it in the form
\begin{equation}
\eqal{
&\varrho u_t-\mu\Delta u-\nu\nabla\divv u=\eta f_s+\varrho_sg-\nabla q-
[\eta(v_{s,t}+(v_s+u)\cdot\nabla v_s)\cr
&\quad+\varrho_su\cdot\nabla v_s]-[\eta(u+v_s)+\varrho_su]\nabla u,\cr}
\label{3.1}
\end{equation}
\begin{equation}
q_t+v\cdot\nabla q=-u\cdot\nabla p_s-\varkappa q\divv(v_s+u)-\varkappa
p_s\divv u,
\label{3.2}
\end{equation}
\begin{equation}
u|_{t=0}=u(0),\ \ q|_{t=0}=q(0)
\label{3.3}
\end{equation}
\begin{equation}
u|_S=0,\ \ v_s|_S=0,
\label{3.4}
\end{equation}
where $\varrho=\varrho_s+\eta$, $v=v_s+u$, $p=p_s+q$.

\noindent
Moreover, it is convenient to consider the equation
\begin{equation}
\eta_t+v\cdot\nabla\eta+\eta\divv v=-\varrho_s\divv u-u\cdot\nabla\varrho_s.
\label{3.5}
\end{equation}
To prove existence of solutions to problem (\ref{3.1})--(\ref{3.5}) we use the
following method of successive approximations. Let $u_n$ be given. Then
$\eta_n$, $q_n$, $u_{n+1}$ are calculated from the following problem
\begin{equation}
\eqal{
&\varrho_nu_{n+1,t}-\mu\Delta u_{n+1}-\nu\nabla\divv u_{n+1}=\eta_nf_s+
\varrho_sg-\nabla q_n\cr
&\quad-[\eta_nv_{s,t}+\eta_n(v_s+u_n)\cdot\nabla v_s+\varrho_su_n\cdot
\nabla v_s]\cr
&\quad-[\eta_n(v_s+u_n)+\varrho_su_n]\nabla u_n,\cr}
\label{3.6}
\end{equation}
\begin{equation}
q_{n,t}+v_n\cdot\nabla q_n=-u_n\cdot\nabla p_s-\varkappa q_n\divv(v_s+u_n)-
\varkappa p_s\divv u_n,
\label{3.7}
\end{equation}
\begin{equation}
\eta_{n,t}+v_s\cdot\nabla\eta_n=-\eta_n\divv(v_s+u_n)-\varrho_s\divv u_n-
u_n\cdot\nabla\varrho_s,
\label{3.8}
\end{equation}
\begin{equation}
u_{n+1}|_{t=0}=u(0),\ \ q_n|_{t=0}=q(0),\ \ \eta_n|_{t=0}=\eta(0),
\label{3.9}
\end{equation}
\begin{equation}
u_{n+1}|_S=0,\ \ v_s|_S=0,
\label{3.10}
\end{equation}
where $v_n=v_s+u_n$, $\varrho_n=\varrho_s+\eta_n$,
$q_n=A(\varrho_n^\varkappa-\varrho_s^\varkappa)$.

\noindent
The $u_0$ approximation is an extension of the initial data $u(0)$.

\noindent
To formulate results of this section we introduce the notation
\begin{equation}
\eqal{
&X_1(u_1,\dots,u_k,t)=\|u_1(t),\dots,u_k(t)\|_{2,\Omega},\cr
&\bar X_1(u_1,\dots,u_k,t)=\|u_1,\dots,u_k\|_{2,\infty,\Omega^t},\cr
&X_2(u_1,\dots,u_k,t)=|u_1(t),\dots,u_k(t)|_{2,1,\Omega},\cr
&\bar X_2(u_1,\dots,u_k,t)=|u_1,\dots,u_k|_{2,1,\infty,\Omega^t},\cr
&\alpha(u)=t^{1/2}\|u\|_{3,2,\Omega^t}.\cr}
\label{3.11}
\end{equation}

\begin{lemma}\label{l3.1}
Assume that $\eta(0)\in H^2(\Omega)$, $\varrho_s\in L_\infty(0,t;H^3(\Omega))$,
$v_s,u_n\in L_2(0,t;H^3(\Omega))$. Then
\begin{equation}
\|\eta_n\|_2^2\le\exp[c\alpha(v_n)][\|\varrho_s\|_{3,\infty,\Omega^t}\alpha(u_n)
+\|\eta(0)\|_2].
\label{3.12}
\end{equation}
Assume that there exists positive constants $\varrho_*$, $\varrho^*$ such that
$\varrho_*\le\varrho(0)\le\varrho^*$.\\
Then for a~sufficiently small time and $\|\eta(0)\|_2$ we have that
\begin{equation}
{1\over2}\varrho_*\le\varrho_n(t)\le2\varrho^*.
\label{3.13}
\end{equation}
\end{lemma}

\begin{proof}
Multiplying (\ref{3.8}) by $\eta_n$ and integrating over $\Omega$ yields
\begin{equation}
{d\over dt}|\eta_n|_2\le c|v_{n,x}|_\infty|\eta_n|_2+c\|\varrho_s\|_1\|u_n\|_2.
\label{3.14}
\end{equation}
Differentiating (\ref{3.8}) with respect to $x$, multiplying by $\eta_{n,x}$
and integrating over $\Omega$, gives
\begin{equation}
{d\over dt}|\eta_{n,x}|_2\le c\|v_n\|_3\|\eta_n\|_1+c\|\varrho_s\|_2\|u_n\|_2.
\label{3.15}
\end{equation}
Differentiating (\ref{3.8}) twice with respect to $x$, multiplying by
$\eta_{n,xx}$ and integrating over $\Omega$ implies
\begin{equation}
{d\over dt}|\eta_{n,xx}|_2\le c\|v_n\|_3\|\eta_n\|_2+c\|\varrho_s\|_3\|u_n\|_3.
\label{3.16}
\end{equation}
Adding (\ref{3.14})--(\ref{3.16}), integrating the result with respect to time,
we obtain (\ref{3.12}). Having (\ref{3.12}) yields (\ref{3.13}). This
concludes the proof.
\end{proof}

\begin{lemma}\label{l3.2}
Assume that $g\in L_2(\Omega^t)$, $f_s\in L_{2,\infty}(\Omega^t)$,
$u(0)\in L_2(\Omega)$, $\varrho_s,\eta_n\in L_\infty(0,t;\Gamma_1^2(\Omega))$,
$v_s,u_n\in L_\infty(0,t;H^2(\Omega))$, $v_{s,t}\in L_{2,\infty}(\Omega^t)$.
Then
\begin{equation}
\eqal{
&|u_{n+1}(t)|_2^2+\|u_{n+1}\|_{1,2,\Omega^t}^2\le c\exp t(\bar X_2(\varrho_s)+
\bar X_2(\eta_n))\cdot\cr
&\quad\cdot[t\|\eta_n\|_{1,\infty,\Omega^t}^2\|f_s\|_{0,\infty,\Omega^t}^2+
\|\varrho_s\|_{1,\infty,\Omega^t}^2\|g\|_{0,2,\Omega^t}^2\cr
&\quad+t\|\eta_n\|_{1,\infty,\Omega^t}^2(1+\|v_{s,t}\|_{0,\infty,\Omega^t}^2)+
t(\bar X_1^2(\varrho_s,v_s)+\bar X_1^2(\eta_n,u_n))^2\bar X_1^2(\eta_n,u_n)\cr
&\quad+|u(0)|_2^2],\cr}
\label{3.17}
\end{equation}
$\bar X_1$, $\bar X_2$ are defined in (\ref{3.11}).
\end{lemma}

\begin{proof}
Multiplying (\ref{3.6}) by $u_{n+1}$, integrating the result over $\Omega$
and applying the H\"older and the Young inequalities, we get
\begin{equation}
\eqal{
&{1\over2}\intop_\Omega\varrho_n\partial_tu_{n+1}^2dx+\|u_{n+1}\|_1^2\le c
[\|\eta_n\|_1^2\|f_s\|_0^2+\|\varrho_s\|_1^2\|g\|_0^2\cr
&\quad+\|\eta_n\|_0^2]+c[\|\eta_n\|_1^2|v_{s,t}|_2^2+\|\eta_n\|_1^2
(\|v_s\|_1^2+\|u_n\|_1^2)\|v_s\|_1^2]\cr
&\quad+c[\|\eta_n\|_1^2(\|u_n\|_1^2+\|v_s\|_1^2)+\|\varrho_s\|_1^2
\|u_n\|_1^2]\|u_n\|_1^2\cr
&\le c[\|\eta_n\|_1^2\|f_s\|_0^2+\|\varrho_s\|_1^2\|g\|_0^2+\|\eta_n\|_0^2]+c
\|\eta_n\|_1^2|v_{s,t}|_2^2\cr
&\quad+c(X_1^2(\varrho_s,v_s)+X_1^2(\eta_n,u_n))^2X_1^2(\eta_n,u_n).\cr}
\label{3.18}
\end{equation}
The first term on the l.h.s. of (\ref{3.18}) equals
$$
{1\over2}{d\over dt}\intop_\Omega\varrho_nu_{n+1}^2dx-\intop_\Omega
\varrho_{n,t}u_{n+1}^2dx,
$$
where the second term is bounded by
$$
\varepsilon|u_{n+1}|_6^2+c/\varepsilon(|\varrho_{s,t}|_3^2+|\eta_{n,t}|_3^2)
|u_{n+1}|_2^2.
$$Employing the consideration in (\ref{3.18}), using that $\varrho_n$ is
separeted from zero (see Lemma \ref{l3.1}) we obtain after integration with
respect to time inequality (\ref{3.17}). This concludes the proof.
\end{proof}

\begin{lemma}\label{l3.3}
Assume that $\varrho_s,v_s,\eta_n,u_n\in L_\infty(0,t;\Gamma_1^2(\Omega))$,
$v_{s,tt}\in L_\infty(0,t;\break L_2(\Omega))$,
$f_s\in L_\infty(0,t;\Gamma_0^1(\Omega))$, $g\in L_2(0,t;\Gamma_0^1(\Omega))$,
$u_t(0)\in L_2(\Omega)$. Then
\begin{equation}
\eqal{
&|u_{n+1,t}(t)|_2^2+\|u_{n+1,t}\|_{1,2,\Omega^t}^2\le c\exp
[t(|\varrho_s|_{2,1,\infty,\Omega^t}^2+|\eta_n|_{2,1,\infty,\Omega^t}^2)]\cr
&\quad\cdot c\big[t|\eta_n|_{2,1,\infty,\Omega^t}^2(1+|f_s|_{1,0,\infty,\Omega^t}^2)+
|\varrho_s|_{2,1,\infty,\Omega^t}^2|g|_{1,0,2,\Omega^t}^2\cr
&\quad+ct\big[|\eta_n|_{2,1,\infty,\Omega^t}^2\|v_{s,tt}\|_{0,\infty,\Omega^t}^2+
\bar X_2^2(\varrho_s,v_s)\bar X_2^2(\eta_n,u_n)\cr
&\quad+(\bar X_2^2(\varrho_s,v_s)\bar X_2^2(\eta_n,u_n)+\bar X_2^4(\varrho_s,v_s)+\bar X_2^4(\eta_n,u_n))\bar X_2^2(\eta_nu_n)\big]\cr
&\quad+|u_{,t}(0)|_2^2\big].\cr}
\label{3.19}
\end{equation}
\end{lemma}

\begin{proof}
Differentiating (\ref{3.6}) with respect to $t$, multiplying the result by
$u_{n+1,t}$ and integrating over $\Omega$ yields
\begin{equation}
\eqal{
&\intop_\Omega(\varrho_{n,t}u_{n+1,t}+\varrho_nu_{n+1,tt}-Au_{n+1,t})\cdot
u_{n+1,t}dx\cr
&=\intop_\Omega(\eta_nf_s+\varrho_sg)\cdot u_{n+1,t}dx-\intop_\Omega\nabla
q_{n,t}\cdot u_{n+1,t}dx\cr
&\quad-\intop_\Omega[\eta_nv_{s,t}+\eta_n(v_s+u_n)\nabla v_s+\varrho_su_n\cdot
\nabla v_s]_{,t}\cdot u_{n+1,t}dx\cr
&\quad-\intop_\Omega([\eta_n(v_s+u_n)+\varrho_su_n]\nabla u_n)_{,t}\cdot
u_{n+1,t}dx.\cr}
\label{3.20}
\end{equation}
First we estimate the l.h.s. of (\ref{3.20}). The first term equals
$$
\bigg|\intop_\Omega\varrho_{n,t}u_{n+1,t}^2dx\bigg|\le\varepsilon
|u_{n+1,t}|_6^2+c/\varepsilon(\|\varrho_{s,t}\|_1^2+\|\eta_{n,t}\|_1^2)
|u_{n+1,t}|_2^2,
$$
the second implies
$$
{1\over2}\intop_\Omega\varrho_n{\partial\over\partial t}|u_{n+1,t}|^2dx=
{1\over2}{d\over dt}\intop_\Omega\varrho_nu_{n+1,t}^2dx-{1\over2}\intop_\Omega
\varrho_{n,t}u_{n+1,t}^2dx,
$$
where the last integral is estimated above. Finally, the last term on the l.h.s. equals
$$
\|u_{n+1,t}\|_1^2.
$$
The first term on the r.h.s. is bounded by
$$
\varepsilon|u_{n+1,t}|_6^2+c/\varepsilon(|\eta_n|_{2,1}^2|f_s|_{1,0}^2+
|\varrho_s|_{2,1}^2|g|_{1,0}^2)
$$
the second by
$$
\varepsilon|u_{n+1,xt}|_2^2+c/\varepsilon|\eta_{n,t}|_2^2
$$
and the last two by
$$\eqal{
&\varepsilon|u_{n+1,t}|_6^2+c/\varepsilon[|\eta_n|_{2,1}^2|v_{s,tt}|_2^2+
|\eta_n|_{2,1}|^2\|v_{s,t}\|_1^2+|\eta_n|_{2,1}^2(|v_s|_{2,1}^2+|u_n|_{2,1}^2)|v_s|_{2,1}^2\cr
&\quad+|\varrho_s|_{2,1}^2|u_n|_{2,1}^2|v_s|_{2,1}^2+|\eta_n|_{2,1}^2
(|v_s|_{2,1}^2+|u_n|_{2,1}^2)|u_n|_{2,1}^2+|\varrho_s|_{2,1}^2|u_n|_{2,1}^4]\cr
&\le\varepsilon|u_{n+1,t}|_6^2+c/\varepsilon[|\eta_n|_{2,1}^2|v_{s,tt}|_2^2+
|\eta_n|_{2,1}^2|v_s|_{2,1}^2+(X_2^4(\varrho_s,v_s)\cr
&\quad+X_2^2(\varrho_s,v_s)
X_2^2(\eta_n,u_n)+X_2^4(\eta_n,u_n))X_2^2(\eta_n,u_n)].\cr}
$$
Using the above estimates in (\ref{3.20}) and assuming that $\varepsilon$ is
sufficiently small, we obtain
\begin{equation}
\eqal{
&{d\over dt}\intop_\Omega\varrho_nu_{n+1,t}^2dx+\|u_{n+1,t}\|_1^2\le c
(\|\varrho_{s,t}\|_1^2+\|\eta_{n,t}\|_1^2)|u_{n+1,t}|_2^2\cr
&\quad+c(|\eta_{n,t}|_2^2+|\eta_n|_{2,1}^2|f_s|_{1,0}^2+|\varrho_s|_{2,1}^2
|g|_{1,0}^2)\cr
&\quad+c[|\eta_n|_{2,1}^2|v_{s,tt}|_2^2+X_2^2(\varrho_s,v_s)X_2^2(\eta_n,u_n)+
(X_2^4(\varrho_s,v_s)\cr
&\quad+X_2^2(\varrho_s,v_s)X_2^2(\eta_n,u_n)+X_2^4(\eta_n,u_n))X_2^2(\eta_n,u_n)].\cr}
\label{3.21}
\end{equation}
in view of(\ref{3.13}) and integration (\ref{3.21}) with respect to time we
get (\ref{3.19}). This concludes the proof.
\end{proof}

\begin{lemma}\label{l3.4.}
Assume that $u_{n+1}\in L_2(0,t;H^3(\Omega))$,
$\varrho_s,\eta_n,u_n\in L_\infty(0,t;\break\Gamma_1^2(\Omega))$,
$v_s\in L_\infty(0,t;\Gamma_0^2(\Omega))$,
$f_s\in L_\infty(0,t;\Gamma_0^1(\Omega))$, $g\in L_2(0,t;\Gamma_0^1(\Omega))$,
$u(0)\in L_2(\Omega)$, $u_t(0)\in H^1(\Omega)$.\\
Then
\begin{equation}
\eqal{
&\|u_{n+1,t}(t)\|_1^2+\|u_{n+1,t}\|_{2,2,\Omega^t}^2\le\varepsilon
\|u_{n+1}\\_{3,2,\Omega^t}^2\cr
&\quad+\varphi\bigg({1\over\varepsilon},t^a\bar X_2(\eta_n,u_n),
\bar X_2(\varrho_s,v_s),|f_s|_{1,0,\infty,\Omega^t},
|v_s|_{2,0,\infty,\Omega^t})\cdot\cr
&\quad\cdot[t^a\bar X_2^2(\eta_n,u_n)+|g|_{1,0,2,\Omega^t}^2+|u(0)|_2^2+
\|u_{,t}(0)\|_1^2],\cr}
\label{3.22}
\end{equation}
where $a>0$ and $\varepsilon\in(0,1)$.
\end{lemma}

\begin{proof}
Differentiate (\ref{3.6}) with respect to time, multiply the result by\break
$-Au_{n+1,t}$ and integrate over $\Omega$. Then we have
\begin{equation}
\eqal{
&-\intop_\Omega(\varrho_{n,t}u_{n+1,t}+\varrho_nu_{n+1,tt})\cdot Au_{n+1,t}dx+
\intop_\Omega|Au_{n+1,t}|^2dx\cr
&=-\intop_\Omega(\eta_nf_s+\varrho_sg)_{,t}\cdot Au_{n+1,t}dx+\intop_\Omega
\nabla q_{n,t}\cdot Au_{n+1,t}dx\cr
&\quad+\intop_\Omega[\eta_nv_{s,t}+\eta_n(v_s+u)\cdot\nabla v_s+\varrho_su_n
\cdot\nabla v_s]_{,t}\cdot Au_{n+1,t}dx\cr
&\quad+\intop_\Omega([\eta_n(v_s+u_n)+\varrho_su_n]\nabla u_n)_{,t}\cdot
Au_{n+1,t}dx.\cr}
\label{3.23}
\end{equation}
Now we estimate the terms from (\ref{3.23}). The first part of the first
integral on the l.h.s. of (\ref{3.23}) is bounded by
$$
\bigg|\intop_\Omega\varrho_{n,t}u_{n+1,t}\cdot Au_{n+1,t}dx\bigg|\le\varepsilon
|Au_{n+1,t}|_2^2+c/\varepsilon|\varrho_{n,t}|_6^2|u_{n+1,t}|_3^2\equiv I_1.
$$
Using (\ref{2.16}) we estimate the second term in $I_1$ by
$$
\varepsilon_1|u_{n+1,xt}|_2^2+c\varepsilon_1|\varrho_{n,t}|_6^4|u_{n+1,t}|_2^2.
$$
In view of the above estimates, application of the H\"older and the Young
inequalities to the integrals from the r.h.s. of (\ref{3.23}) and using the
estimate
\begin{equation}
\|u\|_2\le c\|f\|_0
\label{3.24}
\end{equation}
which holds for solutions to the problem
\begin{equation}
\eqal{
&Au=f,\cr
&u|_S=0,\cr}
\label{3.25}
\end{equation}
we obtain
\begin{equation}
\eqal{
&-\intop_\Omega\varrho_nu_{n+1,tt}\cdot Au_{n+1,t}dx+\|u_{n+1,t}\|_2^2\le
c\|\varrho_{n,t}\|_1^2|u_{n+1,t}|_2^2\cr
&\quad+c|(\eta_nf_s+\varrho_sg)_{,t}|_2^2+c|\nabla q_{n,t}|_2^2\cr
&\quad+c|(\eta_nv_{s,t}+\eta_n(v_s+u_n)\cdot\nabla v_s+\varrho_su_n\cdot
\nabla v_s)_{,t}|_2^2\cr
&\quad+c|[(\eta_n(v_s+u_n)+\varrho_su_n)\cdot\nabla u_n]_{,t}|_2^2.\cr}
\label{3.26}
\end{equation}
Now, we estimate terms from (\ref{3.26}). Using the integration by parts
formula (see Remark \ref{r2.1}) the first integral on the l.h.s. of
(\ref{3.26}) takes the form
\begin{equation}
\intop_\Omega A^{1/2}\varrho_nu_{n+1,tt}A^{1/2}u_{n+1,t}dx+\intop_\Omega
\varrho_nA^{1/2}u_{n+1,tt}A^{1/2}u_{n+1,t}dx.
\label{3.27}
\end{equation}
The second integral in (\ref{3.27}) equals
$$\eqal{
&{1\over2}\intop_\Omega\varrho_n\partial_t|A^{1/2}u_{n+1,t}|^2dx={1\over2}
{d\over dt}\intop_\Omega\varrho_n|A^{1/2}u_{n+1,t}|^2dx\cr
&\quad-{1\over2}\intop_\Omega\varrho_{n,t}|A^{1/2}u_{n+1,t}|^2dx\equiv I_1.\cr}
$$
the second integral in $I_1$ is bounded by
$$
\varepsilon|A^{1/2}u_{n+1,t}|_6^2+c/\varepsilon|\varrho_{n,t}|_2^2
|A^{1/2}u_{n+1,t}|_3^2\equiv I_2.
$$
Using interpolation (\ref{2.16}) we get
$$
I_2\le\varepsilon\|u_{n+1,t}\|_2^2+\varepsilon_1c/\varepsilon
\|u_{n+1,t}\|_2^2+c/\varepsilon\varepsilon_1|\varrho_{n,t}|_2^4|u_{n+1,t}|_2^2.
$$
The first integral in (\ref{3.27}) is estimated by
\begin{equation}
\eqal{
&\varepsilon|u_{n+1,tt}|_2^2+c/\varepsilon|\varrho_{n,x}|_6^2
|u_{n+1,xt}|_3^2\le\varepsilon|u_{n+1,tt}|_2^2+\varepsilon_1
\|u_{n+1,t}\|_2^2\cr
&\quad+c/\varepsilon\varepsilon_1\|\varrho_n\|_2^4|u_{n+1,t}|_2^2.\cr}
\label{3.28}
\end{equation}
To estimate the first term in (\ref{3.28}) we calculate $u_{n+1,tt}$ from
(\ref{3.6}). Hence, we have
\begin{equation}
\eqal{
u_{n+1,t}&={1\over\varrho_n}Au_{n+1}+{1\over\varrho_n}(\eta_nf_s+\varrho_sg)-
{1\over\varrho_n}\nabla q_n-{1\over\varrho_n}[\eta_nv_{s,t}\cr
&\quad+\eta_n(v_s+u_n)\cdot\nabla v_s+
\varrho_su_n\cdot\nabla v_s]-{1\over\varrho_n}[\eta_n(v_s+u_n)\cr
&\quad+\varrho_su_n]\cdot\nabla u_n.\cr}
\label{3.29}
\end{equation}
Continuing, we get
\begin{equation}
\eqal{
u_{n+1,tt}&=\bigg({1\over\varrho_n}Au_{n+1}\bigg)_{,t}+\bigg({1\over\varrho_n}
(\eta_nf_s+\varrho_sg)\bigg)_t-\bigg({1\over\varrho_n}\nabla q_n\bigg)_t\cr
&\quad-\bigg({1\over\varrho_n}[\eta_nv_{s,t}+\eta_n(v_s+u_n)\cdot\nabla v_s+
\varrho_su_n\cdot\nabla v_s]\bigg)_{,t}\cr
&\quad-\bigg({1\over\varrho_n}[\eta_n(v_s+u_n)+\varrho_su_n]\cdot\nabla u_n\bigg)_{,t}.\cr}
\label{3.30}
\end{equation}
Taking, the $L_2$-norm of (\ref{3.30}) ahnd using (\ref{3.13}) we obtain
\begin{equation}
\eqal{
|u_{n+1,tt}|_2^2&\le c\|u_{n+1,t}\|_2^2+c|\varrho_{n,t}|_6^2|Au_{n+1}|_3^2\cr
&\quad+\varphi(|\varrho_s|_{2,1},|\eta_n|_{2,1})[|\eta_n|_{2,1}^2|f_s|_{1,0}^2+
|g|_{1,0}^2+|\eta_n|_{2,1}^2]\cr
&\quad+\varphi(X_2(\varrho_s,v_s),X_2(\eta_n,u_n))[|v_s|_{2,0}^2+X_2^2(\varrho_s,v_s)\cr
&\quad+X_2^2(\eta_n,u_n)]X_2^2(\eta_n,u_n).\cr}
\label{3.31}
\end{equation}
Similarly, the r.h.s. of (\ref{3.26}) is estimatede by the last two terms on
the r.h.s. of (\ref{3.21}).

Using the above estimates in (\ref{3.26}) and assuming that $\varepsilon$,
$\varepsilon_1$ are sufficiently small we obtain
\begin{equation}
\eqal{
&{d\over dt}\intop_\Omega\varrho_n|A^{1/2}u_{n+1,t}|^2dx+\|u_{n+1,t}\|_2^2\cr
&\le\varepsilon_2\|u_{n+1}\|_3^2+c|\varrho_n|_{2,1}^4(|u_{n+1,t}|_2^2+
c/\varepsilon_2|u_{n+1}|_2^2)\cr
&\quad+\varphi(|\varrho_s|_{2,1},|\eta_n|_{2,1})[|\eta_n|_{2,1}^2|f_s|_{1,0}^2+
|g|_{1,0}^2+|\eta_n|_{2,1}^2]\cr
&\quad+\varphi(X_2(\varrho_s,v_s),X_2(\eta_n,u_n))[|v_s|_{2,0}^2+X_2^2(\varrho_s,v_s)\cr
&\quad+X_2^2(\eta_n,u_n)]X_2^2(\eta_n,u_n).\cr}
\label{3.32}
\end{equation}
Using notation (\ref{3.11}) we express (\ref{3.17}) in the form
\begin{equation}
\eqal{
&|u_{n+1}(t)|_2^2+\|u_{n+1}\|_{1,2,\Omega^t}^2\le\varphi(t^a\bar X_2
(\eta_n,u_n),\bar X_2(\varrho_s,v_s),\cr
&\quad \|f_s\|_{0,\infty,\Omega^t})[t^a\bar X_2^2(\eta_n,u_n)+
\|g\|_{0,2,\Omega^t}^2+|u(0)|_2^2].\cr}
\label{3.33}
\end{equation}
Similarly (\ref{3.19}) gives
\begin{equation}
\eqal{
&|u_{n+1,t}(t)|_2^2+\|u_{n+1,t}\|_{1,2,\Omega^t}^2\le\varphi(t^a\bar X_2
(\eta_n,u_n),\bar X_2(\varrho_s,v_s),\cr
&\quad |f|_{1,0,\infty,\Omega^t})[t^a\bar X_2^2(\eta_n,u_n)+
|g|_{1,0,2,\Omega^t}^2+|u_t(0)|_2^2],\cr}
\label{3.34}
\end{equation}
where $a>0$. Integrating (\ref{3.32}) with respect to time and using (\ref{3.33}) and
(\ref{3.34}) we obtain (\ref{3.22}). This concludes the proof.
\end{proof}

\noindent
To estimate the first norm on the r.h.s. of (\ref{3.22}) we need

\begin{lemma}\label{l3.5}
Assume that
\begin{equation}
\eqal{
&u_{n+1,t}\in L_2(0,t;H^2(\Omega)),\quad
\varrho_s,v_s,\eta_n,u_n\in L_\infty(0,t;\Gamma_1^2(\Omega)),\cr
&f_s\in L_\infty(0,t;H^1(\Omega)),\quad g\in L_2(0,t;H^1(\Omega)),\quad
u(0)\in\Gamma_1^2(\Omega).\cr}
\label{3.35}
\end{equation}
Then
\begin{equation}
\eqal{
&\|u_{n+1}(t)\|_2^2+\|u_{n+1}\|_{3,2,\Omega^t}^2\le\varepsilon
\|u_{n+1,t}\|_{2,2,\Omega^t}^2\cr
&\quad+\exp(t\bar X_2(\varrho_s,\eta_n)){1\over\varepsilon}[\varphi(t^a\bar X_2
(\eta_n,u_n),\bar X_2(\varrho_s,v_s),\|f_s\|_{1,\infty,\Omega^t})\cdot\cr
&\quad\cdot t^a\bar X_1(\eta_n,u_n)(1+\|v_{s,t}\|_{0,\infty,\Omega^t}^2)+
\|\varrho_s\|_{1,\infty,\Omega^t}^2\|g\|_{1,2,\Omega^t}^2+|u(0)|_{2,1}^2].\cr}
\label{3.36}
\end{equation}
\end{lemma}

\begin{proof}
To prove the lemma we have to use the local considerations. We restrict them
to neighborhoods near the boundary, because estimates in interior
subdomains are simpler and similar. In view of notation in Section \ref{s2}
we transform (\ref{3.6}) in accordance to (\ref{2.5}). Then we have
\begin{equation}
\eqal{
&\hat\varrho_n\tilde u_{n+1,t}-A_z\tilde u_{n+1}=\hat\eta_n\tilde f_s+
\hat\varrho_s\tilde g-\hat\nabla\hat q_n\hat\zeta\cr
&\quad-[\hat\eta_n\hat v_{s,t}+\hat\eta_n(\hat v_s+\hat u_n)\cdot\hat\nabla
\hat v_s+\hat\varrho_s\hat u_n\cdot\hat\nabla\hat v_s]\hat\zeta\cr
&\quad-[\hat\eta_n(\hat v_s+\hat u_n)+\hat\varrho_s\hat u_n]\cdot\hat\nabla\hat u_n
\hat\zeta+(\hat A_x-A_z)\tilde u_{n+1}\cr
&\quad-[\mu(2\hat\nabla_k\hat u_{n+1}\hat\nabla_k\hat\zeta+\hat u_{n+1}
\hat\Delta\hat\zeta)+\nu(\hat\divv\hat u_{n+1}\hat\nabla\hat\zeta\cr
&\quad+\hat\nabla\hat u_k\hat\nabla_k\hat\zeta+\hat u_k\hat\nabla\hat\nabla_k
\hat\zeta)],\cr}
\label{3.37}
\end{equation}
where $A_z=\mu\Delta_z+\nu\nabla_z\divv_z$,
$\hat A_x=\mu\hat\nabla_x^2+\nu\hat\nabla_x\hat\divv_x$

\noindent
Differentiating (\ref{3.37}) with respect to $\tau$, multiplying the result by
$-A_z\tilde u_{n+1,\tau}$, integrating over $\hat\Omega$ and by parts, we get
\begin{equation}
\eqal{
&\intop_{\hat\Omega}A_z^{1/2}(\hat\varrho_n\tilde u_{n+1,t})_{,\tau}
A_z^{1/2}\tilde u_{n+1,\tau}dz+\intop_{\hat\Omega}
|A_z\tilde u_{n+1,\tau}|^2dz\cr
&=\intop_{\hat\Omega}[\hat\eta_n\tilde f_s+\hat\varrho_s\tilde g]_{,\tau}
A_z\tilde u_{n+_1,\tau}dz+\intop_{\hat\Omega}
(\hat\nabla\hat q_n\hat\zeta)_{,\tau}\cdot A_z\tilde u_{n+1,\tau}dz\cr
&\quad-\intop_{\hat\Omega}([\hat\eta_n\hat v_{s,t}+\hat\eta_n(\hat v_s+
\hat u_n)\cdot\hat\nabla\hat v_s+\hat\varrho_s\hat u_n\cdot\hat\nabla\hat v_s]
\hat\zeta)_{,\tau}A_z\tilde u_{n+1,\tau}dz\cr
&\quad-\intop_{\hat\Omega}([\hat\eta_n(\hat v_s+\hat u_n)+\hat\varrho_s\hat u_n]
\cdot\hat\nabla\hat u_n\hat\zeta)_{,\tau}A_z\tilde u_{n+1,\tau}dz\cr
&\quad+\intop_{\hat\Omega}[(\hat A_x-A_z)\tilde u_{n+1}]_{,\tau}\cdot
A_z\tilde u_{n+1,\tau}dz\cr
&\quad-\intop_{\hat\Omega}[\mu(2\hat\nabla_k\hat u_{n+1}\hat\nabla_k\hat\zeta+
\hat u_{n+1}\hat\Delta\hat\zeta)+\nu(\hat\divv\hat u_{n+1}\hat\nabla\hat\zeta+
\hat\nabla\hat u_{n+1,k}\hat\nabla_k\hat\zeta\cr
&\quad+\hat u_{n+1,k}\hat\nabla\hat\nabla_k\hat\zeta)]_{,\tau}\cdot A_z
\tilde u_{n+1,\tau}dx.\cr}
\label{3.38}
\end{equation}
The first term on the l.h.s. of (\ref{3.38}) equals
$$\eqal{
&\intop_{\hat\Omega}\hat\varrho_nA^{1/2}\tilde u_{n+1,\tau t}\cdot A^{1/2}
\tilde u_{n+1,\tau}dx+\intop_{\hat\Omega}[\hat\varrho_{n,\tau}A^{1/2}
\tilde u_{n+1,t}\cr
&\quad+A^{1/2}\hat\varrho_n\tilde u_{n+1,\tau t}+A^{1/2}\hat\varrho_{n,\tau}
\hat u_{n+1,t}]\cdot A^{1/2}\tilde u_{n+1,\tau}dz\equiv I,\cr}
$$
where the first term in $I$ takes the form
$$\eqal{
I_1&\equiv{1\over2}\intop_{\hat\Omega}\hat\varrho_n\partial_t
|A^{1/2}\tilde u_{n+1,\tau}|^2dz\cr
&={1\over2}{d\over dt}\intop_{\hat\Omega}\hat\varrho_n
|A^{1/2}\tilde u_{n+1,\tau}|^2dz-{1\over2}\intop_{\hat\Omega}\hat\varrho_{n,t}
|A^{1/2}\tilde u_{n+1,\tau}|^2dz,\cr}
$$
and the second term in $I_1$ is bounded by
$$\eqal{
&c|\hat\varrho_{n,t}|_3|A^{1/2}\tilde u_{n+1,\tau}|_3^2\le c
\|\hat\varrho_{n,t}\|_1\|\tilde u_{n+1,\tau}\|_2\|\tilde u_{n+1,\tau}\|_1\cr
&\le\varepsilon\|\tilde u_{n+1,\tau}\|_2^2+c/\varepsilon
\|\hat\varrho_{n,t}\|_1^2\|\tilde u_{n+1,\tau}\|_1^2\cr
&\le\varepsilon\|\tilde u_{n+1,\tau}\|_2^2+{c\over\varepsilon}
\|\hat\varrho_{n,t}\|_1^2\|\tilde u_{n+1,\tau}\|_2\|\tilde u_{n+1,\tau}\|_0\cr
&\le\varepsilon\|\tilde u_{n+1,\tau}\|_2^2+c/\varepsilon
\|\hat\varrho_{n,t}\|_1^4\|\tilde u_{n+1,\tau}\|_0^2.\cr}
$$
We bound the second term in $I$ by
$$
\varepsilon\|\tilde u_{n+1,\tau}\|_2^2+c/\varepsilon\|\hat\varrho_n\|_2^2
\|\tilde u_{n+1,t}\|_1^2.
$$
Next, we examine the terms from the r.h.s. of (\ref{3.38}). The first is
bounded by
$$
\varepsilon\|A_z\tilde u_{n+1,\tau}\|_0^2+c/\varepsilon
(\|\hat\eta_n\|_{2,\hat\Omega}^2\|\tilde f_s\|_1^2+
\|\hat\varrho_s\|_{2,\hat\Omega}^2\|\tilde g\|_1^2),
$$
the second by
$$
\varepsilon\|A_z\tilde u_{n+1,\tau}\|_0^2+c/\varepsilon\|\hat q_n\|_2^2,
$$
the third by
$$\eqal{
&\varepsilon\|A_z\tilde u_{n+1,\tau}\|_0^2+c/\varepsilon
(\|\hat\eta_n\|_{2,\hat\Omega}^2|\hat v_s|_{2,1,\hat\Omega}^2+
\|\hat\eta_n\|_{2,\hat\Omega}^2(\|\hat v_s\|_{2,\hat\Omega}^2+
\|\hat u_n\|_{2,\hat\Omega}^2)\cr
&\quad+\|\hat\varrho_s\|_{2,\hat\Omega}^2\|\hat u_n\|_{2,\hat\Omega}^2
\|\hat v_s\|_{2,\hat\Omega}^2),\cr}
$$
the fourth by
$$
\varepsilon\|A_z\tilde u_{n+1,\tau}\|_0^2+c/\varepsilon
[\|\hat\eta_n\|_{2,\hat\Omega}^2(\|\hat v_s\|_{2,\hat\Omega}^2+
\|\hat u_n\|_{2,\hat\Omega}^2)+\|\hat\varrho_s\|_{2,\hat\Omega}^2
\|\hat u_n\|_{2,\hat\Omega}^2]\cdot\|\hat u_n\|_{2,\hat\Omega}^2.
$$
Finally, the last but one term by
$$
c\lambda\|\tilde u_{n+1,\tau}\|_2^2
$$
and the last by
$$
\varepsilon\|A_z\tilde u_{n+1,\tau}\|_0^2+c/\varepsilon
\|\hat u_{n+1}\|_{2,\hat\Omega}^2.
$$
Employing the above estimates in (\ref{3.38}) and assuming that $\varepsilon$
is sufficiently small, we derive
\begin{equation}
\eqal{
&{d\over dt}\intop_{\hat\Omega}\hat\varrho_n|A_z^{1/2}\tilde u_{n+1,\tau}|^2dz+
|A_z\tilde u_{n+1,\tau}|_2^2\cr
&\le c(\|\hat\varrho_n\|_{2,\hat\Omega}^2\|\tilde u_{n+1,t}\|_1^2+
\|\hat\varrho_{n,t}\|_1^4\|\tilde u_{n+1,\tau}\|_0^2\cr
&\quad+\|\hat\eta_n\|_{2,\hat\Omega}^2\|\tilde f_s\|_1^2+
\|\hat\varrho_s\|_{2,\hat\Omega}^2\|\tilde g\|_1^2+
\|\hat\eta_n\|_{2,\hat\Omega}^2)\cr
&\quad+\varphi(X_1(\hat\eta_n,\hat u_n),X_2(\hat\varrho_s,\hat v_s))
X_1^2(\hat\eta_n,\hat u_n)+c\lambda\|\tilde u_{n+1,\tau}\|_2^2\cr
&\quad+c\|\hat u_{n+1}\|_{2,\hat\Omega}^2.\cr}
\label{3.39}
\end{equation}
From (\ref{3.37}) we have
\begin{equation}
\eqal{
&\|\tilde u_{n+1,zzz}\|_0^2\le c\|\tilde u_{n+1,zz\tau}\|_0^2+c
[\|\hat\varrho_n\|_{2,\hat\Omega}^2\|\tilde u_{n+1,t}\|_1^2\cr
&\quad+\|\hat\eta_n\|_{2,\hat\Omega}^2\|\tilde f_s\|_1^2+
\|\hat\varrho_s\|_{2,\hat\Omega}^2\|\tilde g\|_1^2+
\|\hat\eta_n\|_{2,\hat\Omega}^2]\cr
&\quad+\varphi(X_1(\hat\eta_n,\hat u_n),X_2(\hat\varrho_s,\hat v_s))
X_1^2(\hat\eta_n,\hat u_n)+c\lambda\|\tilde u_{n+1}\|_3^2\cr
&\quad+c\|\hat u_{n+1}\|_{2,\hat\Omega}^2.\cr}
\label{3.40}
\end{equation}
Integrating (\ref{3.39}) and (\ref{3.40}) with respect to time, adding, using
that $\lambda$ is sufficietnly small, using estimates (\ref{3.17}) and
(\ref{3.19}), we obtain
\begin{equation}
\eqal{
&\intop_{\hat\Omega}\hat\varrho_n|\tilde u_{n+1,\tau z}|^2dz+
\|\tilde u_{n+1}\|_{3,2,\Omega^t}^2\cr
&\le\varphi(t^a\bar X_1(\hat\eta_n,\hat u_n),\bar X_2(\hat\varrho_s,\hat v_s),
\|\tilde f_s\|_{1,\infty,\Omega^t})t^a\bar X_1^2(\hat\eta_n,\hat u_n)\cr
&\quad+c\|\hat\varrho_s\|_{2,\hat\Omega}^2\|\tilde g\|_1^2+c
\|\tilde u(0)\|_2^2+c\|\tilde u_t(0)\|_0^2+
c\|\hat u_{n+1}\|_{2,2,\hat\Omega^t}^2.\cr}
\label{3.41}
\end{equation}
Passing in neighborhoods near the boundary to the old variables $x$, deriving
similar estimates in interior subdomains, summing up over all neighborhoods
of the partition of unity and using the expression
$$
{d\over dt}\intop_\Omega u_{n+1,xx}^2dx\le\varepsilon\|u_{n+1,t}\|_2^2+
c/\varepsilon\|u_{n+1,xx}\|_0^2
$$
integrated with respect to time we obtain the inequality
\begin{equation}
\eqal{
&\|u_{n+1}(t)\|_2^2+\|u_{n+1}\|_{3,2,\Omega^t}^2\le\varepsilon
\|u_{n+1,t}\|_{2,2,\Omega^t}^2\cr
&\quad+c/\varepsilon\|u_{n+1,xx}\|_{0,2,\Omega^t}^2+\varphi(t^a\bar X_2
(\eta_n,u_n),\bar X_2(\varrho_s,v_s),\|f_s\|_{1,\infty,\Omega^t})\cdot\cr
&\quad\cdot t^a\bar X_1^2(\eta_n,u_n)+c\|\varrho_s\|_{2,\infty,\Omega^t}^2
\|g\|_{1,2,\Omega^t}^2+c|u(0)|_{2,1}^2,\cr}
\label{3.42}
\end{equation}
where $a>0$. Using the interpolation inequality (\ref{2.17}) in the second norm on the
r.h.s. of (\ref{3.42}) we obtain
\begin{equation}
\eqal{
&\|u_{n+1}(t)\|_2^2+\|u_{n+1}\|_{3,2,\Omega^t}^2\le\varepsilon
\|u_{n+1,t}\|_{2,2,\Omega^t}^2+c/\varepsilon\|u_{n+1}\|_{0,2,\Omega^t}^2\cr
&\quad+\varphi(t^a\bar X_2(\eta_n,u_n),\bar X_2(\varrho_s,v_s),
\|f_s\|_{1,\infty,\Omega^t})t^a\bar X_1^2(\eta_n,u_n)\cr
&\quad+c\|\varrho_s\|_{2,\infty,\Omega^t}^2\|g\|_{1,2,\Omega^t}^2+c|u(0)|_{2,1}^2.\cr}
\label{3.43}
\end{equation}
We use (\ref{3.17}) to estimate the second norm on the r.h.s. of (\ref{3.43}).
Then (\ref{3.36}) follows. This concludes the proof.
\end{proof}

\begin{lemma}\label{l3.6}
Assume that
\begin{equation}
\eqal{
&\varrho_s,v_s\in L_\infty(0,t;\Gamma_1^2(\Omega))\cap L_\infty(0,t;H^3(\Omega)),
\quad v_s\in L_\infty(0,t;\Gamma_0^2(\Omega)),\cr
&f_s\in L_\infty(0,t;\Gamma_0^1(\Omega)),\quad
g\in L_2(0,t;\Gamma_0^1(\Omega)),\quad \eta(0)\in H^2(\Omega),\cr
&u(0)\in\Gamma_1^2(\Omega),\quad t\le T.\cr}
\label{3.44}
\end{equation}
Then there exists $T_*$ sufficietntly small and $M$ which depends on the norms
from (\ref{3.44}) that
\begin{equation}
|u_n(t)|_{2,1}+|u_n|_{3,2,2,\Omega^t}\le M,\quad \forall n\in\N,\ \ t\le T_*.
\label{3.45}
\end{equation}
\end{lemma}

\begin{proof}
From (\ref{3.22}) and (\ref{3.36}) and for $\varepsilon$ sufficiently small
we derive the inequality
\begin{equation}
\eqal{
&|u_{n+1}(t)|_{2,1}^2+|u_{n+1}|_{3,2,2,\Omega^t}^2\le\varphi(t^a\bar X_2
(\eta_n,u_n),\bar X_2(\varrho_s,v_s),\cr
&|f_s|_{1,0,\infty,\Omega^t},|v_s|_{2,0,\infty,\Omega^t})[t^a\bar X_2^2
(\eta_n,u_n)+|g|_{1,0,2,\Omega^T}^2+|u(0)|_{2,1}^2],\cr}
\label{3.46}
\end{equation}
where $a>0$.

\noindent
From (\ref{3.8}) we have
\begin{equation}
\|\eta_{n,t}\|_1\le c(\|v_s\|_2\|\eta_n\|_2+\|\varrho_s\|_2\|u_n\|_2+
\|\eta_n\|_2\|u_n\|_2).
\label{3.47}
\end{equation}
Employing (\ref{3.47}) in (\ref{3.46}) and using (\ref{3.12}) in the result we
obtain the inequality
\begin{equation}
\eqal{
&|u_{n+1}(t)|_{2,1}^2+|u_{n+1}|_{3,2,2,\Omega^t}^2\le\varphi(\alpha(u_n),t^a
|u_n|_{2,1,\infty,\Omega^t},\alpha(v_s),\cr
&\bar X_2(\varrho_s,v_s),|f_s|_{1,0,\infty,\Omega^t},
|v_s|_{2,0,\infty,\Omega^t})[\|\varrho_s\|_{3,\infty,\Omega^t}\alpha(u_n)\cr
&\quad+\|\eta(0)\|_2^2+|g|_{1,0,2,\Omega^t}^2+|u(0)|_{2,1}^2].\cr}
\label{3.48}
\end{equation}
We assume that approximation $u_0$ is constructed by an extension of the
initial data $u(0)$. Assume that (\ref{3.45}) holds. Then (\ref{3.48}) implies
$$\eqal{
&|u_{n+1}(t)|_{2,1}^2+|u_{n+1}|_{3,2,2,\Omega^t}^2\le\varphi(t^{1/2}M,t^aM,
\alpha(v_s),\bar X_2(\varrho_s,v_s),\cr
&|f_s|_{1,0,\infty,\Omega^t},|v_s|_{2,0,\infty,\Omega^t})
[\|\varrho_s\|_{3,\infty,\Omega^t}t^{1/2}M+|g|_{1,0,2,\Omega^t}^2\cr
&\quad+\|\eta(0)\|_2^2+|u(0)|_{2,1}^2].\cr}
$$
If
$$\eqal{
M^2&>\varphi(0,0,0,\bar X_2(\varrho_s,v_s),|f_s|_{1,0,\infty,\Omega^t},
|v_s|_{2,0,\infty,\Omega^t})[|g|_{1,0,2,\Omega^t}^2\cr
&\quad+\|\eta(0)\|_2^2+|u(0)|_{2,1}^2]\cr}
$$
we obtain for $t$ sufficiently small that
$$
|u_{n+1}(t)|_{2,1}^2+|u_{n+1}|_{3,2,2,\Omega^t}^2\le M^2,
$$
because $\varphi$ is an increasing positive continuous function. This implies
(\ref{3.45}) for all $n$ and concludes the proof.
\end{proof}

\noindent
To prove convergence of the considered sequence we introduce the differences
\begin{equation}
U_n=u_n-u_{n-1},\quad E_n=\eta_n-\eta_{n-1},\quad Q_n=q_n-q_{n-1},
\label{3.49}
\end{equation}
which are solutions to the problems
\begin{equation}
\eqal{
&\varrho_nU_{n+1,t}-\mu\Delta U_{n+1}-\nu\nabla\divv U_{n+1}=-E_nu_{n,t}\cr
&\quad+E_nf_s-\nabla Q_n-[E_nv_{s,t}+E_nv_s\cdot\nabla v_s+E_nv_n\cdot
\nabla v_s\cr
&\quad+\eta_{n-1}U_n\cdot\nabla v_s+\varrho_sU_n\cdot\nabla v_s]-
[E_n(v_s+u_n)\cdot\nabla u_n\cr
&\quad+\eta_{n-1}v_s\cdot\nabla U_n+\eta_{n-1}(U_n\cdot\nabla u_n+u_{n-1}\cdot\nabla U_n)],\cr}
\label{3.50}
\end{equation}
\begin{equation}
\eqal{
&E_{n,t}+v_s\cdot\nabla E_n=-E_n\divv v_s-E_n\divv u_n-\eta_{n-1}\divv U_n\cr
&\quad-\varrho_s\divv U_n-U_n\cdot\nabla\varrho_s,\cr}
\label{3.51}
\end{equation}
and
$$
U_n|_S=0,\quad U_n|_{t=0}=0,\quad E_n|_{t=0}=0,\quad n>0.
$$
By $U_0$, $E_0$ we have extensions of the initial data.

\begin{lemma}\label{l3.7}
Let the assumptions of Lemma \ref{l3.1} hold\\
Then
\begin{equation}
\|U_{n+1}\|_{0,\infty,\Omega^t}+\|U_{n+1}\|_{1,2,\Omega^t}\le\varphi(M)t^a
(\|U_n\|_{0,\infty,\Omega^t}+\|U_n\|_{1,2,\Omega^t}),
\label{3.52}
\end{equation}
where $a>0$.
\end{lemma}

\begin{proof}
Multiplying (\ref{3.50}) by $U_{n+1}$, integrating the result over $\Omega$
and using boundary conditions we have
\begin{equation}
\eqal{
&{1\over2}\intop_\Omega\varrho_n\partial_tU_{n+1}^2dx+\|U_{n+1}\|_1^2\le c
(|f_s|^2+1+\varphi(|u_n|_{2,1},|v_s|_{2,1}))|E_n|_2^2\cr
&\quad+c\|\eta_{n-1}\|_2^2(\|v_s\|_2^2+\|u_n\|_2^2)|U_n|_2^2\le\varphi(M)
(|E_n|_2^2+|U_n|_2^2),\cr}
\label{3.53}
\end{equation}
where we used assumptions of Lemma \ref{l3.6}. Continuing, we have
\begin{equation}
\eqal{
&{d\over dt}\intop_\Omega\varrho_nU_{n+1}^2dx+\|U_{n+1}\|_1^2\le c
(|\varrho_{s,t}|_3^2+|\eta_{n,t}|_3^2)|U_{n+1}|_2^2\cr
&\quad+\varphi(M)(|E_n|_2^2+|U_n|_2^2).\cr}
\label{3.54}
\end{equation}
Multiplying (\ref{3.51}) by $E_n$ and integrating over $\Omega$ we get
$$
{d\over dt}|E_n|_2\le c(\|v_s\|_3+\|u_n\|_3)|E_n|_2+c(\|\eta_{n-1}\|_2+\|\varrho_s\|_2)\|U_n\|_1.
$$
Integrating the above inequality with respect to time yields
\begin{equation}
\eqal{
|E_n(t)|_2&\le\exp[ct^{1/2}(\|v_s\|_{3,2,\Omega^t}+\|u_n\|_{3,2,\Omega^t})]
\cdot\cr
&\quad\cdot c(\|\eta_{n-1}\|_{2,\infty,\Omega^t}+
\|\varrho_s\|_{2,\infty,\Omega^t})t^{1/2}\|U_n\|_{1,2,\Omega^t}.\cr}
\label{3.55}
\end{equation}
Using (\ref{3.55}) in (\ref{3.54}) and integraing the result with respect to
time we get (\ref{3.52}). This concludes the proof.
\end{proof}

\noindent
From Lemma \ref{l3.1}, \ref{l3.2} we have

\begin{theorem}\label{t3.3}
Let the assumptions of Lemma \ref{l3.1} hold. Let $\sigma(0)$ be so small that
(\ref{3.13}) is satisfied. Then for sufficiently small time $T$ there exists
a solution to problem (\ref{3.1})--(\ref{3.5}) such that
$$\eqal{
&u\in L_\infty(0,T;\Gamma_1^2(\Omega))\cap L_2(0,T;\Gamma_2^3(\Omega))\equiv A(\Omega^T),\cr
&\eta\in L_\infty(0,T;\Gamma_1^2(\Omega))\equiv B(\Omega^T)\quad \textsl{and}\cr
&\|u\|_{A(\Omega^T)}+\|\eta\|_{B(\Omega^T)}\le M,\cr}
$$
where $M$ is introduced in Lemma \ref{l3.1}.
\end{theorem}

\section{Differential inequality}\label{s4}
\setcounter{equation}{0}

To prove the differential inequality we need existence of sufficiently regular
local solution. Moreover, we need relations (\ref{1.14}). However, they hold
for the local solution, the differential inequality shows that they remain to
hold for all time.

We remark that in this section we shall indicate any equivalent norm in
$H^s(\Omega)$ with the same symbol $\|\cdot\|_s$. However constants in this
section depend on constants from (\ref{1.14}), we do not mark it.

First, applying global domain considerations, we show

\begin{lemma}\label{l4.1}
For sufficiently regular solutions to (\ref{1.5}), (\ref{1.10})--(\ref{1.12})
it holds
\begin{equation}
\eqal{
&{d\over dt}\bigg\|\sqrt{\varrho_s}u,\sqrt{\varrho}u_t,{1\over\sqrt{p_s}}q,
{1\over\sqrt{p_s}}q_t\bigg\|_0^2+\|u,u_t\|_1^2+\|q,q_t\|_0^2\cr
&\le c\varphi_1\|f_s\|_1^2+cA_1\|g\|_0^2+cA_1(1+A_1^2)\varphi_1+cA_1\varphi_1^2
+c(1+\varphi_1)\varphi_1^2\equiv cX_0^2,\cr}
\label{4.1}
\end{equation}
where
\begin{equation}
\varphi_1(t)=|u(t)|_{2,1}^2+|q(t)|_{2,1}^2,\quad
A_1(t)=|\varrho_s(t)|_{2,1}^2+|v_s(t)|_{2,1}^2.
\label{4.2}
\end{equation}
\end{lemma}

\begin{proof}
Multiplying (\ref{1.5}) by $u$, (\ref{1.10}) by $q$, integrating over $\Omega$,
adding and using $(\ref{1.3})_2$ and (\ref{2.2}), respectively, we get
\begin{equation}
\eqal{
&{d\over dt}\intop_\Omega\bigg(\varrho_su^2+{1\over\varkappa p_s}q^2\bigg)dx+
\|u\|_1^2\le c\bigg|\intop_\Omega\bar f\cdot udx\bigg|+c
\bigg|\intop_\Omega\bar hqdx\bigg|\cr
&\quad+c\intop_\Omega|\divv v_s|q^2dx.\cr}
\label{4.3}
\end{equation}
Differentiating (\ref{1.5}) and (\ref{1.10}) with respect to time, multiplying
the results by $u_t$ and $q_t$, respectively, adding, integrating over
$\Omega$ with using $(\ref{1.3})_2$ and (\ref{2.2}), respectively, we obtain
\begin{equation}
\eqal{
&{d\over dt}\bigg\|\sqrt{\varrho_s}u_t,{1\over\sqrt{p_s}}q\bigg\|_0^2+
\|u_t\|_0^2\le c\bigg|\intop_\Omega(\varrho_{st}u_t^2+(\varrho_sv_s)_{,t}\cdot
\nabla uu_t)dx\bigg|\cr
&\quad+c\bigg|\intop_\Omega\bigg[\bigg({1\over p_s}\bigg)_{,t}q_t^2+
\bigg({v_s\over p_s}\bigg)_{,t}\cdot\nabla qq_t\bigg]dx\bigg|+c
\bigg|\intop_\Omega\divv v_sq_t^2dx\bigg|\cr
&\quad+c\bigg|\intop_\Omega\bar f_t\cdot
u_tdx\bigg|+c\bigg|\intop_\Omega\bar h_t\cdot q_tdx\bigg|\equiv cI_1.\cr}
\label{4.4}
\end{equation}
From (\ref{1.10}) we have
\begin{equation}
\|q_t\|_0^2\le c(\|v_s\cdot\nabla q\|_0^2+\|\divv u\|_0^2+\|\bar h\|_0^2).
\label{4.5}
\end{equation}
Adding (\ref{4.3}) and (\ref{4.5}) appropriately, we get
\begin{equation}
\eqal{
&{d\over dt}\bigg\|\sqrt{\varrho_s}u,{1\over\sqrt{p_s}}q\bigg\|_0^2+\|u\|_1^2+
\|q_t\|_0^2\le c\bigg|\intop_\Omega\bar f\cdot udx\bigg|+c\bigg|\intop_\Omega
\bar h\cdot qdx\bigg|\cr
&\quad+c\bigg|\intop_\Omega|\divv v_s|q^2dx\bigg|+c\|v_s\cdot\nabla q\|_0^2+
c\|\bar h\|_0^2\equiv cI_2.\cr}
\label{4.6}
\end{equation}
Let us consider the Stokes system
\begin{equation}
\eqal{
&-\mu\Delta u+\nabla q=\bar f+\nu\nabla\divv u-\varrho_s(u_t+v_s\cdot\nabla u),\cr
&\divv u=\divv u,\cr
&u|_S=0.\cr}
\label{4.7}
\end{equation}
Let us introduce a~function $\varphi$ such that
\begin{equation}
\divv\varphi=q,\quad \varphi|_S=0.
\label{4.8}
\end{equation}
Following \cite{KP}, there exists a~solution to (\ref{4.8}) such that
$\varphi\in H^1$ and
\begin{equation}
\|\varphi\|_1\le c\|q\|_0.
\label{4.9}
\end{equation}
Multiplying $(\ref{4.7})_1$ by $\varphi$, integrating over $\Omega$ and using
(\ref{4.9}) yields
\begin{equation}
\|q\|_0^2\le c(\|u\|_1^2+|\bar f|_{6/5}^2+|u_t|_{6/5}^2+|v_s\cdot\nabla u|_{6/5}^2),
\label{4.10}
\end{equation}
where we used that $\varrho_*\le\varrho_s\le\varrho^*$. Inequalities (\ref{4.4}) and (\ref{4.6}) imply
\begin{equation}
{d\over dt}\bigg\|\sqrt{\varrho_s}u,\sqrt{\varrho_s}u_t,{1\over\sqrt{p_s}}q,
{1\over\sqrt{p_s}}q_t\bigg\|_0^2+\|u,u_t\|_1^2+\|q_t\|_0^2\le c(I_1+I_2).
\label{4.11}
\end{equation}
Now, adding appropriately (\ref{4.10}) and (\ref{4.11}) gives
\begin{equation}
\eqal{
&{d\over dt}\bigg\|\sqrt{\varrho_s}u,\sqrt{\varrho_s}u_t,{1\over\sqrt{p_s}}q,
{1\over\sqrt{p_s}}q_t\bigg\|_0^2+\|u,u_t\|_1^2+\|q,q_t\|_0^2\le c(I_1+I_2\cr
&\quad+|\bar f|_{6/5}^2+|v_s\cdot\nabla u|_{6/5}^2).\cr}
\label{4.12}
\end{equation}
Now, we calculate
$$\eqal{
I_1&\le\varepsilon(\|u_t\|_1^2+\|q_t\|_0^2)+c/\varepsilon[|\varrho_s|_{2,1}^2+
|v_s|_{2,1}^2+|\varrho_s|_{2,1}^2|v_s|_{2,1}^2\cr
&\quad+|v_s|_{2,1}^2
|\varrho_s|_{2,1}^4](|u|_{2,1}^2+|q|_{2,1}^2)+c/\varepsilon(|\bar f_t|_{6/5}^2+\|\bar h_t\|_0^2),\cr
I_2&\le\varepsilon(\|u\|_1^2+\|q\|_0^2)+c/\varepsilon(\|v_s\|_2^2\|q\|_1^2+
|\bar f|_{6/5}^2+\|\bar h\|_0^2).\cr}
$$
Emploing the above estimates in (\ref{4.12}) yields
\begin{equation}
\eqal{
&{d\over dt}\bigg\|\sqrt{\varrho_s}u,\sqrt{\varrho_s}u_t,{1\over\sqrt{p_s}}q,
{1\over\sqrt{p_s}}q_t\bigg\|_0^2+\|u,u_t\|_1^2+\|q,q_t\|_0^2\cr
&\le cA_1(1+A_1)\varphi_1+c(|\bar f_1|_{6/5}^2+\|\bar h_t\|_0^2+\|\bar h\|_0^2)
+\intop_\Omega\bar f_{2t}\cdot u_tdx.\cr}
\label{4.13}
\end{equation}
Now estimate the last four terms from the r.h.s. of (\ref{4.13}). Using that
$\bar f=\bar f_1+\bar f_2$, $\bar f_2=-\eta u_t$ (see (\ref{4.5})) we have
$$
|\bar f_1|_{6/5}^2\le|\eta|_2^2\|f_s\|_1^2+\|\varrho_s\|_1^2\|g\|_0^2+cA_1
(1+A_1)\varphi_1+cA_1\varphi_1^2+c\varphi_1^3.
$$
Next
$$
\intop_\Omega\bar f_{2t}\cdot u_tdx=-\intop_\Omega(\eta u_t)_tu_tdx=-{1\over2}
{d\over dt}\intop_\Omega\eta u_t^2dx-{1\over2}\intop_\Omega\eta_t u_t^2dx,
$$
where the second integral is estimated by
$$
\bigg|\intop_\Omega\eta_tu_t^2dx\bigg|\le\varepsilon\|u_t\|_1^2+c/\varepsilon
|\eta_t|_2^2\|u_t\|_1^2\le\varepsilon\|u_t\\_1^2+c/\varepsilon\varphi_1^2.
$$
Finally, we have
$$
\|\bar h_t\|_0^2\le c(A_1\varphi_1+\varphi_1^2)
$$
and
$$
\|\bar h\|_0^2\le c(A_1\varphi_1+\varphi_1^2).
$$
Using the above estimates in (\ref{4.13}) and assuming that $\varepsilon$ is
sufficiently small we derive (\ref{4.1}). This concludes the proof.
\end{proof}

Next we obtain an estimate for the second spatial derivatives of $u$ and the
first of $q$.

\begin{lemma}\label{l4.2}
For sufficiently smooth solutions we have
\begin{equation}
{d\over dt}\bigg\|\sqrt{\varrho_s}u,\sqrt{\varrho}u_t,{1\over\sqrt{p_s}}q,
{1\over\sqrt{p_s}}q_t,\sqrt{\varrho_s}u_x,{1\over\sqrt{p_s}}q_x\bigg\|_0^2+
|u|_{2,1}^2+\|q,q_t\|_1^2\le cX_2^2,
\label{4.14}
\end{equation}
where
\begin{equation}
\eqal{
&X_1^2+|\eta|_{1,0}^2|f_s|_{1,0}^2+|\varrho_s|_{1,0}^2|g|_{1,0}^2+
\|\eta\|_1^2\|v_{stt}\|_0^2\cr
&\le c\varphi_1|f_s|_{1,0}^2+cA_1|g|_{1,0}^2+\|\eta\|_1^2\|v_{stt}\|_0^2+
cA_1(1+A_1)\varphi_1+(1+A_1+\varphi_1)\varphi_1^2\cr
&\equiv cX_2,\cr}
\label{4.15}
\end{equation}
and $X_1$ is introduced in (\ref{4.36}).
\end{lemma}

\begin{proof}
Differentiate (\ref{2.5}) and (\ref{2.6}) with respect to $\tau$, next multiply
by $\tilde u_\tau$ and $\tilde q_\tau$, respectively, add, integrate over
$\hat\Omega$ and use transformed equations $(\ref{1.3})_2$ and (\ref{2.2}).
Then we get
\begin{equation}
\eqal{
&{d\over dt}\intop_{\hat\Omega}\bigg(\hat\varrho_s\tilde u_\tau^2+
{1\over\varkappa\hat p_s}\tilde q_\tau^2\bigg)dz+\|\tilde u_{z\tau}\|_0^2+
\|\divv\tilde u_\tau\|_0^2\cr
&=-\intop_{\hat\Omega}(\hat\varrho_{s\tau}\tilde u_t+(\hat\varrho_s
\hat v_s)_{,\tau}\cdot\tilde u_z)\cdot\tilde u_\tau dz-\intop_{\hat\Omega}
\bigg[\bigg({1\over\varkappa\hat p_s}\bigg)_{,\tau}\tilde q_t+
\bigg({\hat v_s\over\varkappa\hat p_s}\bigg)_{,\tau}\tilde q_z\bigg]
\tilde q_\tau dz\cr
&\quad+(\varkappa+1)\intop_{\hat\Omega}{1\over\hat p_s}\hat\divv\hat v_s
\tilde q_\tau^2dz+\intop_{\hat\Omega}[\tilde{\bar f}_{,\tau}\cdot\tilde u_\tau+
k_{1,\tau}\tilde u_\tau+\tilde{\bar h}_{,\tau}\tilde q_\tau+
k_{2,\tau}\tilde q_\tau]dz.\cr}
\label{4.16}
\end{equation}
Let us introduce the quantities
\begin{equation}
\tilde\varphi_1(t)=|\tilde u|_{2,1}^2+|\tilde q|_{2,1}^2,\quad
\hat B_1(t)=\|\hat v_s(t)\|_{2,\hat\Omega}^2+\|\hat p_s(t)\|_{2,\hat\Omega}^2.
\label{4.17}
\end{equation}
Then the first three terms on the r.h.s. of (\ref{4.16}) are bounded by
$$
\varepsilon(\|\tilde u_\tau\|_1^2+\|\tilde q_\tau\|_0^2)+{c\over\varepsilon}
\hat B_1(1+\hat B_1)\tilde\varphi_1.
$$
The last term on the r.h.s. of (\ref{4.16}) is bounded by
$$
\varepsilon(\|\tilde u_{\tau\tau}\|_0^2+\|\tilde q_\tau\|_0^2)+
{c\over\varepsilon}(\|\tilde{\bar f}\|_0^2+\|\tilde{\bar h}\|_1^2)+
{c\over\varepsilon}(\|k_1\|_0^2+\|k_{2,\tau}\|_0^2).
$$
In view of (\ref{2.7}) and (\ref{2.8}) we have
\begin{equation}
\|k_1\|_0^2\le c\lambda(\|\tilde q_z\|_0^2+\|\tilde u_{zz}\|_0^2)+c
(\|\hat q\|_{0,\hat\Omega}^2+\|\hat u\|_{1,\hat\Omega}^2)
\label{4.18}
\end{equation}
and (\ref{2.9}), (\ref{2.10}) imply
\begin{equation}
\|k_{2,\tau}\|_0^2\le c\lambda\|\tilde u_{zz}\|_0^2+c\|\tilde u\|_1^2+
c\hat B_1(1+\hat B_1)\|\tilde q\|_1^2.
\label{4.19}
\end{equation}
Employing the above estimjates in (\ref{4.16}) and assuming that $\varepsilon$
is sufficiently small yields
\begin{equation}
\eqal{
&{d\over dt}\intop_{\hat\Omega}\bigg(\hat\varrho_s\tilde u_\tau^2+
{1\over\varkappa\hat p_s}\tilde q_\tau^2\bigg)dz+\|\tilde u_{z\tau}\|_0^2+
\|\divv\tilde u_\tau\|_0^2\cr
&\le\varepsilon\|\tilde q_\tau\|_0^2+{c\over\varepsilon}\lambda
(\|\tilde u_{zz}\|_0^2+\|\tilde q_z\|_0^2)+{c\over\varepsilon}
(\|\hat u\|_{1,\hat\Omega}^2+\|\hat q\|_{0,\hat\Omega}^2)\cr
&\quad+{c\over\varepsilon}\hat B_1(1+\hat B_1)\tilde\varphi_1+
{c\over\varepsilon}(\|\tilde{\bar f}\|_0^2+\|\tilde{\bar h}\|_1^2).\cr}
\label{4.20}
\end{equation}
To estimate the last two norms we introduce
\begin{equation}
\hat\varphi_1(t)=|\hat u|_{2,1,\hat\Omega}^2+|\hat q|_{2,1,\hat\Omega}^2,\quad
\hat B_2(t)=|\hat v_s(t)|_{2,1,\hat\Omega}^2+
|\hat\varrho_s(t)|_{2,1,\hat\Omega}^2.
\label{4.21}
\end{equation}
Then, using (\ref{2.11}), we have
\begin{equation}
\eqal{
\|\tilde{\bar f}\|_0^2&\le c\|\hat\eta\|_{2,\hat\Omega}^2
\|\tilde f_s\|_0^2+c\|\tilde g\|_0^2+c\hat B_2(1+\hat B_2)\tilde\varphi_1\cr
&\quad+c(1+\hat B_2+\hat\varphi_1)\hat\varphi_1\tilde\varphi_1.\cr}
\label{4.22}
\end{equation}
Finally, (\ref{2.12}) implies
\begin{equation}
\|\tilde{\bar h}\|_1^2\le c(1+\hat B_2)(\hat B_2+\hat\varphi_1)\tilde\varphi_1.
\label{4.23}
\end{equation}
Using estimate (\ref{4.22}) and (\ref{4.23}) in (\ref{4.20}) yields
\begin{equation}
\eqal{
&{d\over dt}\intop_{\hat\Omega}\bigg(\hat\varrho_s\tilde u_\tau^2+
{1\over\varkappa\hat p_s}\tilde q_\tau^2\bigg)dz+\|\tilde u_{z\tau}\|_0^2+
\|\divv\tilde u_\tau\|_0^2\cr
&\le\varepsilon\|\tilde q_\tau\|_0^2+{c\over\varepsilon}\lambda
(\|\tilde u_{zz}\|_0^2+\|\tilde q_z\|_0^2)+{c\over\varepsilon}
(\|\hat u\|_{1,\hat\Omega}^2+\|\hat q\|_{0,\hat\Omega}^2)\cr
&\quad+{c\over\varepsilon}\hat B_2(1+\hat B_2)\tilde\varphi_1+
{c\over\varepsilon}(1+\hat B_2+\hat\varphi_1)\hat\varphi_1\tilde\varphi_1+
c\hat\varphi_1\|\tilde f_s\|_0^2+c\|\tilde g\|_0^2.\cr}
\label{4.24}
\end{equation}
Multiplying the third component of (\ref{2.13}) by $\tilde q_n$ and integrating
the result over $\hat\Omega$ yields
\begin{equation}
\eqal{
&{\mu+\nu\over\varkappa}\intop_{\hat\Omega}\partial_n\bigg({1\over\hat p_s}
\tilde q_t\bigg)\tilde q_ndz+{\mu+\nu\over\varkappa}\intop_{\hat\Omega}
\partial_n\bigg({1\over\hat p_s}\hat v_s\cdot\hat\nabla\tilde q\bigg)
\tilde q_ndz+\|\tilde q_n\|_0^2\cr
&\le c(\|\tilde u_{z\tau}\|_0^2+\|\tilde u_t\|_0^2+\|\hat v_s\|_{1,\hat\Omega}^2\|\tilde u\|_2^2+
\|\tilde{\bar f}\|_0^2+\|k_1\|_0^2)+c(\|\tilde{\bar h}\|_1^2+\|k_2\|_1^2).\cr}
\label{4.25}
\end{equation}
The first two terms on the l.h.s. of (\ref{4.25}) equal
$$\eqal{
&{\mu+\nu\over2\varkappa}\intop_{\hat\Omega}\bigg({1\over\hat p_s}\partial_t
\tilde q_n^2+{1\over\hat p_s}\hat v_s\cdot\hat\nabla\tilde q_n^2\bigg)dz+
{\mu+\nu\over\varkappa}\intop_{\hat\Omega}\bigg[\partial_n
\bigg({1\over\hat p_s}\bigg)\tilde q_t\tilde q_n\cr
&\quad+\partial_n
\bigg({\hat v_s\over\hat p_s}\bigg)\cdot\hat\nabla\tilde q\tilde q_n\bigg]dz\equiv I_1+I_2.\cr}
$$
Employing equation (\ref{2.2}) transformed to variables $z$ in $I_1$ yields
$$
{\mu+\nu\over2\varkappa}{d\over dt}\intop_{\hat\Omega}{1\over\hat p_s}
\tilde q_n^2dz+{(\mu+\nu)(\varkappa+1)\over2\varkappa}\intop_{\hat\Omega}
{1\over\hat p_s}\hat\divv\hat v_s\tilde q_n^2dz,
$$
where the second integral is bounded by
$$
\varepsilon\|\tilde q_n\|_0^2+c/\varepsilon\|\hat v_s\|_{2,\hat\Omega}^2
\|\tilde q_n\|_0^2.
$$
Next,
$$
|I_2|\le\varepsilon\|\tilde q_n\|_0^2+{c\over\varepsilon}
\|\hat p_s\|_{2,\hat\Omega}^2(\|\tilde q_t\|_0^2+\|\hat v_s\|_{2,\hat\Omega}^2
\|\tilde q\|_2^2).
$$
In view of the above relations and estimates (\ref{4.18}), (\ref{4.19}),
(\ref{4.22}), (\ref{4.23}) we obtain from (\ref{4.25}), for sufficiently small
$\varepsilon$, the inequality
\begin{equation}
\eqal{
&{d\over dt}\intop_{\hat\Omega}{1\over\hat p_s}\tilde q_n^2dz+
\|\tilde q_n\|_0^2\le c(\|\tilde u_{z\tau}\|_0^2+\|\tilde u_t\|_0^2)\cr
&\quad+{c\over\varepsilon}\lambda(\|\tilde u_{zz}\|_0^2+
\|\tilde q_z\|_0^2)+{c\over\varepsilon}(\|\hat u\|_{1,\hat\Omega}^2+
\|\hat q\|_{0,\hat\Omega}^2)\cr
&\quad+{c\over\varepsilon}\hat B_2(1+\hat B_2)\tilde\varphi_1+{c\over\varepsilon}
(1+\hat B_2+\hat\varphi_1)\hat\varphi_1\tilde\varphi_1+c\hat\varphi_1
\|\tilde f_s\|_0^2+c\|\tilde g\|_0^2,\cr}
\label{4.26}
\end{equation}
where (\ref{4.2}) and (\ref{4.21}) were used.

Taking the $L_2$-norm of the third component of (\ref{2.14}) yields
\begin{equation}
\eqal{
&\|\divv\tilde u_n\|_0^2\le c(\|\tilde u_{z\tau}\|_0^2+\|\tilde u_t\|_0^2+
\|\tilde q_n\|_0^2)+c\|\hat v_s\|_{2,\hat\Omega}^2\|\tilde u\|_1^2\cr
&\quad+c(\|\tilde{\bar f}\|_0^2+\|k_1\|_0^2).\cr}
\label{4.27}
\end{equation}
Using again (\ref{4.18}) and (\ref{4.22}) and also notation (\ref{4.17}),
(\ref{4.21}) we have
\begin{equation}
\eqal{
&\|\divv\tilde u_n\|_0^2\le c(\|\tilde u_{z\tau}\|_0^2+\|\tilde u_t\|_0^2+
\|\tilde q_n\|_0^2)+c\lambda(\|\tilde u_{zz}\|_0^2+\|\tilde q_z\|_0^2)\cr
&\quad+c(\|\hat u\|_{1,\hat\Omega}^2+\|\hat q\|_{0,\hat\Omega}^2)+c\hat B_2
(1+\hat B_2)\tilde\varphi_1+c(1+\hat B_2+\hat\varphi_1)
\hat\varphi_1\tilde\varphi_1\cr
&\quad+c\hat\varphi_1\|\tilde f_s\|_0^2+c\|\tilde g\|_0^2.\cr}
\label{4.28}
\end{equation}
Adding appropriately (\ref{4.24}), (\ref{4.26}) and (\ref{4.28}) gives
\begin{equation}
\eqal{
&{d\over dt}\intop_{\hat\Omega}\bigg(\hat\varrho_s\tilde u_\tau^2+
{1\over\hat p_s}\tilde q_\tau^2\bigg)dz+{d\over dt}\intop_{\hat\Omega}
{1\over\hat p_s}\tilde q_n^2dz+\|\tilde u_{z\tau}\|_0^2+
\|\divv\tilde u\|_1^2+\|\tilde q_n\|_0^2\cr
&\le\varepsilon\|\tilde q_\tau\|_0^2+c\|\tilde u_t\|_0^2+{c\over\varepsilon}\lambda(\|\tilde u_{zz}\|_0^2+
\|\hat q\|_{0,\hat\Omega}^2)+{c\over\varepsilon}\hat B_2(1+\hat B_2)\tilde\varphi_1\cr
&\quad+{c\over\varepsilon}(1+\hat B_2+\hat\varphi_1)\hat\varphi_1\tilde\varphi_1+
c\hat\varphi_1\|\tilde f_s\|_0^2+c\|\tilde g\|_0^2.\cr}
\label{4.29}
\end{equation}
Passing to the old variables $x$ in (\ref{4.29}), deriving an inequality
similar to (\ref{4.29}) in an interior subdomain, summing the inequalities
over all neighborhoods of the partition of unity, we obtain
\begin{equation}
\eqal{
&{d\over dt}\intop_{\hat\Omega}\bigg(\varrho_su_\tau^2+{1\over p_s}q_x^2\bigg)dx
+\|u_{z\tau}\|_0^2+\|\divv u\|_1^2+\|q_n\|_0^2\cr
&\le\varepsilon\|q_\tau\|_0^2+c\|u_t\|_0^2+{c\over\varepsilon}\lambda
(\|u_{xx}\|_0^2+\|q_x\|_0^2)+{c\over\varepsilon}(\|u\|_1^2+\|q\|_0^2)\cr
&\quad+{c\over\varepsilon}A_1(1+A_1)\varphi_1+{c\over\varepsilon}(1+A_1+\varphi_1)
\varphi_1^2+c\varphi_1\|f_s\|_0^2+c\|g\|_0^2,\cr}
\label{4.30}
\end{equation}
where $u_\tau$ means that in a~neighborhood of the boundary there are only
tangent derivatives. Similarly, $q_n$ means that only normal derivative near
the boundary appears.

For solutions to problem (\ref{4.7}) we have
\begin{equation}
\|u\|_2^2+\|q\|_1^2\le c(\|\divv u\|_1^2+\|\bar f\|_0^2+\|u_t\|_0^2+
\|v_s\|_2^2\|u\|_1^2).
\label{4.31}
\end{equation}
Adding appropriately (\ref{4.30}) and (\ref{4.31}) yields
\begin{equation}
\eqal{
&{d\over dt}\intop_\Omega\bigg(\varrho_su_\tau^2+{1\over p_s}q_x^2\bigg)dx+
\|u\|_2^2+\|q\|_1^2\le c\|u_t\|_0^2+c(\|u\|_1^2+\|q\|_0^2)\cr
&\quad+cA_1(1+A_1)\varphi_1+c(1+A_1+\varphi_1)\varphi_1^2+c\varphi_1
\|f_s\|_0^2+c\|g\|_0^2,\cr}
\label{4.32}
\end{equation}
where we used that $\lambda$ is sufficiently small and
\begin{equation}
\|\bar f\|_0^2\le cA_1(1+A_1)\varphi_1+c(1+A_1+\varphi_1)\varphi_1^2+
c\varphi_1\|f_s\|_0^2+c\|g\|_0^2.
\label{4.33}
\end{equation}
Equation (\ref{1.10}) yields
\begin{equation}
\|q_t\|_1^2\le c(\|\divv u\|_1^2+A_1(1+A_1)\varphi_1+(1+A_1)\varphi_1^2).
\label{4.34}
\end{equation}
Then (\ref{4.32}) and (\ref{4.34}) yield
\begin{equation}
{d\over dt}\intop_\Omega\bigg(\varrho_su_\tau^2+{1\over p_s}q_x^2\bigg)dx+
\|u\|_2^2+\|q,q_t\|_1^2\le c(\|u_t\|_0^2+\|u\|_1^2+\|q\|_0^2)+cX_1^2,
\label{4.35}
\end{equation}
where
\begin{equation}
X_1^2=A_1(1+A_1)\varphi_1+(1+A_1+\varphi_1)\varphi_1^2+\varphi_1
\|f_s\|_0^2+\|g\|_0^2.
\label{4.36}
\end{equation}
To obtain the full derivatives with respect to $x$ of $u$ under the time
derivative in (\ref{4.35}) we multiply (\ref{1.5}) by $-Au$ and integrate
over $\Omega$. Then we have
\begin{equation}
\eqal{
&-\intop_\Omega\varrho_su_tAudx-\intop_\Omega\varrho_s v_s\cdot\nabla uAudx+
\|Au\|_0^2\cr
&\le c(\|q_x\|_0^2+\|\bar f\|_0^2).\cr}
\label{4.37}
\end{equation}
Integrating by parts the first two terms on the l.h.s. of (\ref{4.37}) equals
$$\eqal{
&\intop_\Omega A^{1/2}(\varrho_su_t+\varrho_sv_s\cdot\nabla u)\cdot A^{1/2}udx\cr
&=\intop[A^{1/2}\varrho_su_t+A^{1/2}(\varrho_sv_s)\cdot\nabla u]\cdot
A^{1/2}udx\cr
&\quad+\intop_\Omega(\varrho_sA^{1/2}u_t+\varrho_sv_s\cdot\nabla A^{1/2}u)
A^{1/2}udx\equiv I_1+I_2,\cr}
$$
where
$$
|I_1|\le\varepsilon\|u_x\|_1^2+c/\varepsilon\|\varrho_s\|_2^2(1+\|v_s\|_2^2)
(\|u\|_2^2+\|u_t\|_0^2)
$$
and application of $(\ref{1.3})_2$ in $I_2$ yields
$$
I_2={1\over2}{d\over dt}\intop_\Omega\varrho_s|A^{1/2}u|^2dx.
$$
Then (\ref{4.37}) takes the form
\begin{equation}
{d\over dt}\intop_\Omega\varrho_s|A^{1/2}u|^2dx+\|Au\|_0^2\le c
(\|q_x\|_0^2+X_1^2).
\label{4.38}
\end{equation}
Inequalities (\ref{4.35}) and (\ref{4.38}) imply
\begin{equation}
\eqal{
&{d\over dt}\intop_\Omega\bigg(\varrho_su_x^2+{1\over p_s}q_x^2\bigg)dx+
\|u\|_2^2+\|q,q_t\|_1^2\le c(\|u_t\|_0^2+\|u\|_1^2\cr
&\quad+\|q\|_0^2)+cX_1^2.\cr}
\label{4.39}
\end{equation}
To combine estimates (\ref{4.1}) and (\ref{4.39}) we need to use the estimates
which are written in the more explicit way than in Lemma \ref{l4.1}
\begin{equation}
\eqal{
|f_{1t}|_{6/5}^2&\le c|\eta|_{1,0}^2|f_s|_{1,0}^2+|\varrho_s|_{1,0}^2|g|_{1,0}^2\cr
&\quad+c\|\eta\|_1^2\|v_{stt}\|_0^2+c|v_s|_{2,1}^4|\eta|_{2,1}^2+c|v_s|_{2,1}^2
|\varrho_s|_{2,1}^2|u|_{2,1}^2\cr
&\quad+c|v_s|_{2,1}^2|\eta|_{2,1}^2|u|_{2,1}^2+c|\varrho_s|_{2,1}^2|u|_{2,1}^4
+c|\eta|_{2,1}^2|u|_{2,1}^4.\cr}
\label{4.40}
\end{equation}
Moreover,
\begin{equation}
\intop_\Omega\bar f_{2t}\cdot u_tdx=-{1\over2}{d\over dt}\intop_\Omega\eta
u_t^2dx+{1\over2}\intop\eta_tu_t^2dx,
\label{4.41}
\end{equation}
where the second integral is bounded by
$$
\varepsilon\|u_t\|_1^2+c/\varepsilon\|\eta_t\|_1^2\|u_t\|_1^2.
$$
Next, we have
\begin{equation}
\|\bar h_0\|_0^2\le c(\|p_s\|_2^2+\|v_s\|_2^2)(\|u\|_2^2+\|q\|_2^2)+c(\|u\|_2^2+\|q\|_2^2)^2
\label{4.42}
\end{equation}
and
\begin{equation}
\|\bar h_t\|_0^2\le c(1+|p_s|_{2,1}^2)[|p_s|_{2,1}^2|u|_{2,1}^2+|v_s|_{2,1}^2
|q|_{2,1}^2+|q|_{2,1}^2|u|_{2,1}^2].
\label{4.43}
\end{equation}
Employing (\ref{4.40})--(\ref{4.43}) in (\ref{4.1}) yields
\begin{equation}
\eqal{
&{d\over dt}\bigg\|\sqrt{\varrho_s}u,\sqrt{\varrho}u_t,{1\over\sqrt{p_s}}q,
{1\over\sqrt{p_s}}q_t\bigg\|_0^2+\|u,u_t\|_1^2+\|q,q_t\|_0^2\cr
&\le cA_1(1+A_1)\varphi_1+c(1+A_1)\varphi_1^2+c|\eta|_{1,0}|^2|f_s|_{1,0}^2+
c|\varrho_s|_{1,0}^2|g|_{1,0}^2\cr
&\quad+c\|\eta\|_1^2\|v_{stt}\|_0^2.\cr}
\label{4.44}
\end{equation}
From (\ref{4.44}) and (\ref{4.39}) we obtain (\ref{4.14}). This concludes the
proof of Lemma \ref{l4.2}.
\end{proof}

Next we formulate the lemma describing higher regularity of time derivative
of $u$.

\begin{lemma}\label{l4.3}
For sufficiently smooth solutions we have
\begin{equation}
\eqal{
&{d\over dt}\bigg\|\sqrt{\varrho_s}u,\sqrt{\varrho}u_t,{1\over\sqrt{p_s}}q,
{1\over\sqrt{p_s}}q_t,\sqrt{\varrho_s}u_x,\sqrt{\varrho}u_{xt},
{1\over\sqrt{p_s}}q_x\|_0^2\cr
&\quad+\|u,u_t\|_2^2+\|q,q_t\|_1^2\le\varepsilon\|u_{xxx}\|_0^2+c/\varepsilon
X_3^2,\cr}
\label{4.45}
\end{equation}
where
\begin{equation}
\eqal{
X_3^2&=\varphi_1|f_s|_{1,0}^2+A_1|g|_{1,0}^2+\|\eta\|_2^2\|v_{stt}\|_0^2+
A_1(1+A_1+A_1^3)\varphi_1\cr
&\quad+A_1\varphi_1^2+(1+\varphi_1+\varphi_1^3)\varphi_1^2.\cr}
\label{4.46}
\end{equation}
\end{lemma}

\begin{proof}
To prove the lemma we consider (\ref{1.5}) in the form
\begin{equation}
\eqal{
&\varrho u_t+\varrho v\cdot\nabla u-\mu\Delta u-\nu\nabla\divv u+\nabla q=
\eta f_s+\varrho_sg\cr
&\quad-[\eta(v_{st}+(v_s+u)\cdot\nabla v_s)+\varrho_su\cdot\nabla v_s]
\equiv\bar f_3.\cr}
\label{4.47}
\end{equation}
Differentiating (\ref{4.47}) with respect to time, multiplying by $-Au_t$ and
integrating over $\Omega$ one obtains
\begin{equation}
\eqal{
&-\intop_\Omega\varrho_tu_t\cdot Au_tdx-\intop_\Omega\varrho u_{tt}\cdot
Au_tdx-\intop_\Omega(\varrho v)_t\cdot\nabla u\cdot Au_tdx\cr
&\quad-\intop_\Omega\varrho v\cdot\nabla u_t\cdot Au_tdx+\|Au_t\|_0^2\le c
(\|\nabla q_t\|_0^2+\|\bar f_{st}\|_0^2).\cr}
\label{4.48}
\end{equation}
The first term on the l.h.s. of the above inequality is bounded by
$$
\varepsilon\|Au_t\|_0^2+c/\varepsilon\|\varrho_tu_t\|_0^2\le\varepsilon
\|Au_t\|_0^2+c/\varepsilon(A_1+\varphi_1)\varphi_1,
$$
and the third by
$$
\varepsilon\|Au_t\|_0^2+c/\varepsilon(\|\varrho_tv\cdot\nabla u\|_0^2+
\|\varrho v_t\cdot\nabla u\|_0^2)\le\varepsilon\|Au_t\|_0^2+c/\varepsilon
(A_1+\varphi_1)^2\varphi_1.
$$
Integrating by parts in the second and the fourth terms on the l.h.s. of
(\ref{4.48}), we derive
$$\eqal{
I&=\sum_{\alpha=1}^2\bigg(-\intop_\Omega A_\alpha^{1/2}(\varrho u_{tt}\cdot
A_\alpha^{1/2}u_t)dx+\intop_\Omega A_\alpha^{1/2}(\varrho u_{tt})\cdot
A_\alpha^{1/2}u_tdx\cr
&\quad-\intop_\Omega A_\alpha^{1/2}(\varrho v\cdot\nabla u_t\cdot
A_\alpha^{1/2}u_t)dx+\intop_\Omega A_\alpha^{1/2}(\varrho v\cdot\nabla u_t)
\cdot A_\alpha^{1/2}u_tdx)\cr
&\equiv\sum_{i=1}^4I_i.\cr}
$$
The operator $A=\mu\Delta+\nu\nabla\divv$, so $A_1^{1/2}=\sqrt{\mu}\nabla$,
$A_2^{1/2}=\sqrt{\nu}\divv$. Therefore,
$I_1=-\intop_\Omega[\mu\nabla\cdot(\varrho u_{tt}\cdot\nabla u_t)+\nu\divv
(\varrho u_{tt}\divv u_t)]dx$, so the Green formula and $u|s=0$ imply that
$I_1=0$. Next
$$\eqal{
I_3&=-\intop_\Omega\mu\nabla\cdot(\varrho v\cdot\nabla u_t\cdot\nabla u_t)dx-
\nu\intop_\Omega\divv(\varrho v\cdot\nabla u_t\divv u_t)dx\cr
&=-\mu\intop_S\varrho v\cdot\nabla u_t\cdot\bar n\cdot\nabla u_tdS-
\nu\intop_S\varrho v\cdot\nabla v_t\cdot\bar n\divv u_tdS.\cr}
$$
Since $v|_S=0$ it follows that $I_3=0$. therefore $I=I_2+I_4$. Continuing,
we have
$$\eqal{
I&=\sum_{\alpha=1}^2\bigg[\intop_\Omega A_\alpha^{1/2}\varrho u_{tt}
A_\alpha^{1/2}u_tdx+\intop_\Omega\varrho A_\alpha^{1/2}u_{tt}\cdot
A_\alpha^{1/2}u_tdx\cr
&\quad+\intop_\Omega A_\alpha^{1/2}(\varrho v)\cdot\nabla u_t\cdot
A_\alpha^{1/2}u_tdx+\intop_\Omega\varrho v\cdot\nabla A_\alpha^{1/2}u_t
A_\alpha^{1/2}u_tdx\bigg]\cr
&\equiv\sum_{i=1}^4J_i.\cr}
$$
Using the equation of continuity $(\ref{1.1})_2$ yields
$$\eqal{
J_2+J_4&={1\over2}\sum_{\alpha=1}^2\intop_\Omega(\varrho\partial_t
|A_\alpha^{1/2}u_t|^2+\varrho v\cdot\nabla|A_\alpha^{1/2}u_t|^2)dx\cr
&={1\over2}\sum_{\alpha=1}^2{d\over dt}\intop_\Omega\varrho
|A_\alpha^{1/2}u_t|^2dx={1\over2}{d\over dt}\intop_\Omega\varrho(\mu u_{xt}^2+
\nu|\divv u_t|^2)dx.\cr}
$$
Next
$$
|J_1|\le\varepsilon\|u_{tt}\|_0^2+{c\over\varepsilon}(|\varrho_{sx}|_6^2+
|\eta_x|_6^2)|u_{xt}|_3^2\equiv J'_1.
$$
Employing the interpolation inequality (\ref{2.16}) we obtain
\begin{equation}
J'_1\le\varepsilon\|u_{tt}\|_0^2+{c\over\varepsilon}\varepsilon_1
\|u_{xt}\|_1^2+{c\over\varepsilon}\varepsilon_1^{-1}(\|\varrho_{sx}\|_1^4+
\|\eta_x\|_1^4)\|u_{xt}\|_0^2,
\label{4.49}
\end{equation}
where the last term in (\ref{4.49}) is bounded by
${c\over\varepsilon}\varepsilon^{-1}(A_1^2+\varphi_1^2)\varphi_1$. Finally,
$$
|J_3|\le c\intop_\Omega(|\varrho_x|\,|v|+|\varrho|\,|v_x|)|v_{xt}|^2dx\le c
(|\varrho_x|_3|v|_\infty+|\varrho|_\infty|v_x|_3)|u_{xt}|_3^2\equiv J'_3.
$$
Applying again (\ref{2.16}) yields
$$
J'_3\le\varepsilon\|u_{xt}\|_1^2+c/\varepsilon(\|\varrho_s\|_2^4+
\|\eta\|_2^4)(\|v_s\|_2^4+\|\eta\|_2^4)\|u_{xt}\|_0^2,
$$
where the second term is bounded by
$$
{c\over\varepsilon}(A_1^2+\varphi_1^2)^2\varphi_1.
$$
Employing the above estimate in (\ref{4.48}) gives
\begin{equation}
\eqal{
&{d\over dt}\intop_\Omega\varrho\sum_{\alpha=1}^2|A_\alpha^{1/2}u_t|^2dx+
\|u_t\|_2^2\le\varepsilon\|u_{tt}\|_0^2+c(\|\nabla q_t\|_0^2+
\|\bar f_{3t}\|_0^2)\cr
&\quad+c(A_1+A_1^2+A_1^4)\varphi_1+(\varphi_1+\varphi_1^2+\varphi_1^4)
\varphi_1.\cr}
\label{4.50}
\end{equation}
To end the proof we have to estimate $\|\bar f_{st}\|_0$ and $\|u_{tt}\|_0$.
From (\ref{4.47}) we derive
$$\eqal{
\|\bar f_{3t}\|_0^2&\le c(|\eta|_{2,1}^2|f_s|_{1,0}^2+|\varrho_s|_{2,1}^2
|g|_{1,0}^2+\|\eta\|_2^2\|v_{stt}\|_0^2\cr
&\quad+A_1\varphi_1+A_1(A_1+\varphi_1)\varphi_1).\cr}
$$
Calculating $u_t$ from (\ref{4.47}) and differentaiting the result with
respect to $t$ yields
$$\eqal{
u_{tt}&=-{1\over\varrho}[-\varrho v\cdot\nabla u+\mu\Delta u+\nu\nabla\divv u-
\nabla q+\bar f_3]_{,t}\cr
&\quad-{1\over\varrho^2}\varrho_t[-\varrho v\cdot\nabla u+\mu\Delta u+
\nu\nabla\divv u-\nabla q+\bar f_3].\cr}
$$
Hence, we calculate
$$\eqal{
\|u_{tt}\|_0^2&\le c(\|(\varrho v\cdot\nabla u)_t\|_0^2+\|u_t\|_2^2+
\|q_{xt}\|_0^2+\|\bar f_{3,t}\|_0^2)\cr
&\quad+c\|\varrho_t(|v\cdot\nabla u|+|u_{xx}|+|q_x|+|\bar f_3|)\|_0^2\cr
&\le c(\|u_t\|_2^2+\|q_t\|_1^2+\|\bar f_{3,t}\|_0^2)+c
\|(\varrho v\cdot\nabla u)_t\|_0^2+c\|\varrho_tu_{xx}\|_0^2\cr
&\quad+c\|\varrho_tq_x\|_0^2+c\|\varrho_t\bar f_3\|_0^2,\cr}
$$
where
$$\eqal{
\|(\varrho v\cdot\nabla u)_{,t}\|_0^2
&\le c(A_1+\varphi_1)\varphi_1+c(A_1+\varphi_1)^2\varphi_1,\cr
\|\varrho_tu_{xx}\|_0^2
&\le|\varrho_t|_6^2|u_{xx}|_3^2\le\varepsilon\|u_{xxx}\|_0^2+c/\varepsilon
\|\varrho_t\|_1^4\|u_{xx}\|_0^2\cr
&\le\varepsilon\|u_{xxx}\|_0^2+c/\varepsilon(A_1+\varphi_1)^2\varphi_1,\cr}
$$
where (\ref{2.16}) is used. Continuing
$$
\|\varrho_tq_x\|_0^2\le\|\varrho_t\|_1^2\|q_x\|_1^2\le(A_1+\varphi_1)\varphi_1
$$
and
$$\eqal{
\|\varrho_t\bar f_3\|_0^2&\le\|\varrho_t\|_1^2\|\bar f_3\|_1^2\le
c(A_1+\varphi_1)[\varphi_1\|f_s\|_1^2+A_1\|g\|_1^2\cr
&\quad+A_1^2\varphi_1+A_1\varphi_1^2].\cr}
$$
Summarizing, we have
$$\eqal{
&\|u_{tt}\|_0^2\le\varepsilon\|u_{xxx}\|_0^2+c(\|u_t\|_2^2+\|q_t\|_1^2+
\|\bar f_{3,t}\|_0^2)\cr
&\quad+{c\over\varepsilon}[(A_1+\varphi_1)^2\varphi_1+(A_1+\varphi_1)\varphi_1]
+c\varphi_1\|f_s\|_1^2+cA_1\|g\|_1^2+c\|v_{st}\|_1^2\varphi_1.\cr}
$$
In view of the above estimates inequality (\ref{4.50}) takes the form
\begin{equation}
\eqal{
&{d\over dt}\|\sqrt{\varrho}u_{xt}\|_0^2+\|u_t\|_2^2\le\varepsilon
\|u_{xxx}\|_0^2+c\|q_{xt}\|_0^2+c(\varphi_1|f_s|_{1,0}^2\cr
&\quad+|\varrho_s|_{1,0}^2|g|_{1,0}^2+\|\eta\|_2^2\|v_{stt}\|_0^2+
(A_1+A_1^2+A_1^4)\varphi_1\cr
&\quad+A_1\varphi_1^2+(\varphi_1+\varphi_1^2+\varphi_1^4)\varphi_1).\cr}
\label{4.51}
\end{equation}
Adding appropriately (\ref{4.14}) and (\ref{4.51}) we obtain (\ref{4.45}).
This concludes the proof.
\end{proof}

Finally, we pass to estimate the third spacial derivatives of velocity and
the second spacial derivatives of density.

\begin{lemma}\label{l4.4}
For sufficiently regular solutions we have
\begin{equation}
\eqal{
&{d\over dt}\intop_\Omega\bigg(\varrho_su_{xx}^2+{1\over p_s}q_{xx}^2)dx+
\|u\|_3^2+\|q\|_2^2\le\varepsilon\|u_{xxt}\|_0^2\cr
&\quad+c(\|u_t\|_0^2+\|u\|_2^2+\|q\|_1^2)+cX_4^2,\cr}
\label{4.52}
\end{equation}
where
\begin{equation}
\eqal{
X_4^2&=\varphi_1\|f_s\|_1^2+A_1\|g\|_1^2+(1+A_1)A_2\varphi_1+(1+A_2)
\varphi_1\Phi_1\cr
&\quad+(A_1+\varphi_1)^2\varphi_1+A_1(1+\varphi_1)\varphi_1.\cr}
\label{4.53}
\end{equation}
\end{lemma}

\begin{proof}
Differentiating (\ref{2.5}) and (\ref{2.6}) twice with respect to $\tau$,
multiplying by $\tilde u_{\tau\tau}$ and $\tilde q_{\tau\tau}$, respectively,
integrating the results over $\hat\Omega$, adding and integrating by parts,
we get
\begin{equation}
\eqal{
&{1\over2}{d\over dt}\intop_{\hat\Omega}\bigg(\hat\varrho_s
\tilde u_{\tau\tau}^2+{1\over p_s}\tilde q_{\tau\tau}^2\bigg)dz+
\|\tilde u_{\tau\tau}\|_1^2+\|\divv\tilde u_{\tau\tau}\|_0^2\cr
&\le\intop_{\hat\Omega}(\hat\varrho_{s\tau\tau}\tilde u_t+\hat\varrho_{s\tau}
\tilde u_{\tau t})\cdot\tilde u_{\tau\tau}dz+\intop_{\hat\Omega}
((\hat\varrho_s\hat v_s)_{,\tau\tau}\cdot\hat\nabla\tilde u+
(\hat\varrho_s\hat v_s)_{,\tau}\hat\nabla\tilde u_\tau)\tilde u_{\tau\tau}dz\cr
&\quad+\intop_{\hat\Omega}\bigg[\bigg({1\over\hat p_s}\bigg)_{,\tau\tau}
\tilde q_t+\bigg({1\over\hat p_s}\bigg)_{,\tau}\tilde q_{\tau t}\bigg]
\tilde q_{\tau\tau}dz\cr
&\quad+\intop_{\hat\Omega}\bigg[\bigg({\hat v_s\over\hat p_s}\bigg)_{,\tau\tau}
\cdot\hat\nabla\tilde q+\bigg({\hat v_s\over\hat p_s}\bigg)_{,\tau}\cdot
\hat\nabla\tilde q_\tau\bigg]\tilde q_{\tau\tau}dz\cr
&\quad+\intop_{\hat\Omega}{(\varkappa+1)\varkappa\over\hat p_s}\hat\divv
\hat v_s\tilde q_{\tau\tau}^2dz+\intop_{\hat\Omega}\tilde{\bar f}_{,\tau\tau}
\cdot\tilde u_{\tau\tau}dz+\intop_{\hat\Omega}\tilde{\bar h}_{,\tau\tau}\cdot
\tilde u_{\tau\tau}dz\cr
&\quad+\intop_{\hat\Omega}k_{1,\tau\tau}\tilde u_{\tau\tau}dz+
\intop_{\hat\Omega}k_{2,\tau\tau}\tilde q_{,\tau\tau}dz.\cr}
\label{4.54}
\end{equation}
Now we estimate the terms from the r.h.s. of (\ref{4.54}). We bound the first
term by
$$
\varepsilon_1\|\tilde u_{\tau\tau}\|_1^2+c/\varepsilon
\|\hat\varrho_{s\tau}\|_{1,\hat\Omega}^2\|\tilde u_t\|_1^2,
$$
the second by
$$
\varepsilon_2\|\tilde u_{\tau\tau}\|_1^2+c\varepsilon_2
\|\hat\varrho_s\|_{2,\hat\Omega}^2\|\hat v_s\|_{2,\hat\Omega}^2
\|\tilde u\|_2^2,
$$
the third by
$$
\varepsilon_3\|\tilde q_{\tau\tau}\|_0^2+c/\varepsilon_3
(1+\|\hat p_s\|_{2,\hat\Omega}^2)\|\hat p_s\|_{3,\hat\Omega}^2\|\tilde q_t\|_1^2,
$$
the fourth by
$$
\varepsilon_4\|\tilde q_{\tau\tau}\|_0^2+c/\varepsilon_4
(\|\hat v_s\|_{3,\hat\Omega}^2+\|\hat v_s\|_{2,\hat\Omega}^2
\|\hat p_s\|_{3,\hat\Omega}^2+\|\hat v_s\|_{2,\hat\Omega}^2
\|\hat p_s\|_{2,\hat\Omega}^2)\|\tilde q\|_2^2,
$$
the fifth by
$$
\varepsilon_5\|\tilde q_{\tau\tau}\|_0^2+c/\varepsilon_5
\|\hat v_s\|_{3,\hat\Omega}^2\|\tilde q_{\tau\tau}\|_0^2.
$$
Next we examine
$$
\bigg|\intop_{\hat\Omega}\tilde{\bar f}_{,\tau\tau}\cdot\tilde u_{\tau\tau}
dz\bigg|=\bigg|\intop_{\hat\Omega}\tilde{\bar f}_{,\tau}\tilde u_{\tau\tau\tau}
dz\le\varepsilon_6\|\tilde u_{\tau\tau\tau}\|_0^2+c/\varepsilon_6
\|\tilde{\bar f}\|_1^2
$$
and
$$\eqal{
&\bigg|\intop_{\hat\Omega}k_{1,\tau\tau}\tilde u_{\tau\tau}dz\bigg|=
\bigg|\intop_{\hat\Omega}k_{1,\tau}\tilde u_{\tau\tau\tau}dz\bigg|\le
\varepsilon_7\|\tilde u_{\tau\tau\tau}\|_0^2\cr
&\quad+{c\over\varepsilon_7}\lambda(\|\tilde u\|_3^2+\|\tilde q\|_2^2)+c
(\|\hat u\|_{2,\hat\Omega}^2+\|\hat q\|_{1,\hat\Omega}^2).\cr}
$$
Using the above estimates in (\ref{4.54}) and assuming that
$\varepsilon_1-\varepsilon_7$ are sufficiently small, we get
\begin{equation}
\eqal{
&{d\over dt}\intop_{\hat\Omega}\bigg(\hat\varrho_s\tilde u_{\tau\tau}^2+
{1\over\hat p_s}\tilde q_{\tau\tau}^2\bigg)dz+\|\tilde u_{\tau\tau}\|_1^2+
\|\divv\tilde u_{\tau\tau}\|_0^2\cr
&\le\varepsilon\|\tilde q_{\tau\tau}\|_0^2+{c\over\varepsilon}\lambda
(\|\tilde u\|_3^2+\|\tilde q\|_2^2)+{c\over\varepsilon}
(\|\hat u\|_{2,\hat\Omega}^2+\|\hat q\|_{1,\hat\Omega}^2)\cr
&\quad+{c\over\varepsilon}
[\|\hat\varrho_s\|_{2,\hat\Omega}^2(1+\|\hat v_s\|_{2,\hat\Omega}^2)+(1+\|\hat p_s\|_{2,\hat\Omega}^2)\|\hat p_s\|_{3,\hat\Omega}^2+\|\hat v_s\|_{3,\hat\Omega}^2\cr
&\quad+(\|\hat p_s\|_{3,\hat\Omega}^2+
\|\hat p_s\|_{2,\hat\Omega}^4)\|\hat v_s\|_{2,\hat\Omega}^2]
(|\tilde u|_{2,1}^2+|\tilde q|_{2,1}^2)\cr
&\quad+{c\over\varepsilon}\|\tilde{\bar f}\|_1^2+c\bigg|\intop_{\hat\Omega}
\tilde{\bar h}_{\tau\tau}\cdot\tilde u_{\tau\tau}dz\bigg|+c
\bigg|\intop_{\hat\Omega}k_{2,\tau\tau}\tilde q_{\tau\tau}dz\bigg|.\cr}
\label{4.55}
\end{equation}
Finally, we have to estimate the last two terms on the r.h.s. of (\ref{4.55}). First we examine
$$\eqal{
&\bigg|\intop_{\hat\Omega}\tilde{\bar h}_{,\tau\tau}\tilde q_{\tau\tau}dz\bigg|
\le\bigg|\intop_{\hat\Omega}\bigg({\hat u\over\hat p_s}\hat\nabla
(\hat p_s+\hat q)\hat\zeta\bigg)_{,\tau\tau}\tilde q_{\tau\tau}dz\bigg|\cr
&\quad+\bigg|\intop_{\hat\Omega}\bigg({\hat q\over\hat p_s}\hat\divv
(\hat v_s+\hat u)\hat\zeta\bigg)_{,\tau\tau}\tilde q_{\tau\tau}dz\bigg|
\equiv I_1+I_2.\cr}
$$
Continuing,
$$\eqal{
I_1&\le\bigg|\intop_{\hat\Omega}\bigg({\hat u\over\hat p_s}\hat\nabla\hat p_s
\hat\zeta\bigg)_{,\tau\tau}\tilde q_{\tau\tau}dz\bigg|+
\bigg|\intop_{\hat\Omega}\bigg({\hat u\over\hat p_s}\hat\nabla
\tilde q\bigg)_{,\tau\tau}\tilde q_{\tau\tau}dz\bigg|\cr
&\quad+\bigg|\intop_{\hat\Omega}\bigg({\hat u\over\hat p_s}\hat\nabla
\hat\zeta\hat q\bigg)_{,\tau\tau}\tilde q_{\tau\tau}dz\bigg|\equiv
I_1^1+I_1^2+I_1^3.\cr}
$$
First, we estimate
$$
I_1^1\le\varepsilon\|\tilde q_{\tau\tau}\|_0^2+c/\varepsilon
\|\hat p_s\|_{3,\hat\Omega}^2(1+\|\hat p_s\|_{2,\hat\Omega}^2)\|\tilde u\|_2^2.
$$
Next, we examine
$$\eqal{
I_1^2&=\bigg|\intop_{\hat\Omega}\bigg({\hat u\over\hat p_s}\bigg)_{,\tau\tau}
\hat\nabla\tilde q\tilde q_{\tau\tau}\bigg|+\bigg|\intop_{\hat\Omega}
\bigg({\hat u\over\hat p_s}\bigg)_{,\tau}(\hat\nabla\tilde q)_{,\tau}
\tilde q_{\tau\tau}\bigg|+\bigg|\intop_{\hat\Omega}{\hat u\over\hat p_s}
\hat\nabla\tilde q_{\tau\tau}\tilde q_{\tau\tau}dz\bigg|\cr
&\equiv J_1+J_2+J_3,\cr}
$$
where
$$
J_1+J_2\le\varepsilon\|\tilde q_{\tau\tau}\|_0^2+c/\varepsilon
(1+\|\hat p_s\|_{3,\hat\Omega}^2)\|\hat u\|_{3,\hat\Omega}^2\|\tilde q\|_2^2.
$$
Finally, integrating by parts in $J_3$ yields
$$
J_3=\bigg|{1\over2}\intop_{\hat\Omega}\hat\nabla\bigg({\hat u\over\hat p_s}\bigg)
\tilde q_{\tau\tau}^2dz\bigg|\le\varepsilon\|\tilde q_{\tau\tau}\|_0^2+
{c\over\varepsilon}(1+\|\hat p_s\|_{3,\hat\Omega}^2)\|\hat u\|_{3,\hat\Omega}^2
\|\tilde q_{\tau\tau}\|_0^2.
$$
Continuing,
$$
I_1^3\le\varepsilon\|\tilde q_{\tau\tau}\|_0^2+c/\varepsilon
(1+\|\hat p_s\|_{2,\hat\Omega}^2)\|\hat u\|_{2,\hat\Omega}^2
\|\hat q\|_{2,\hat\Omega}^2.
$$
Similarly, we have
$$
I_2\le\varepsilon\|\tilde q_{\tau\tau}\|_0^2+c/\varepsilon
(1+\|\hat p_s\|_{2,\hat\Omega}^2)(\|\hat v_s\|_{3,\hat\Omega}^2+
\|\hat u\|_{3,\hat\Omega}^2)\|\hat q\|_{2,\hat\Omega}^2.
$$
Summarizing the above considerations, one has
\begin{equation}
\eqal{
&\bigg|\intop_{\hat\Omega}\tilde{\bar h}_{\tau\tau}\tilde q_{\tau\tau}
dz\bigg|\le\varepsilon\|\tilde q_{\tau\tau}\|_0^2+c/\varepsilon
[(1+\|\hat p_s\|_{2,\hat\Omega}^2)(\|\hat p_s\|_{3,\hat\Omega}^2
\|\hat u\|_{2,\hat\Omega}^2\cr
&\quad+\|\hat v_s\|_{3,\hat\Omega}^2\|\hat q\|_{2,\hat\Omega}^2)+
(1+\|\hat p_s\|_{3,\hat\Omega}^2)\|\hat q\|_{2,\hat\Omega}^2
\|\hat u\|_{3,\hat\Omega}^2]\cr
&\le\varepsilon\|\tilde q_{\tau\tau}\|_0^2+{c\over\varepsilon}[(1+\hat A_1)
\hat A_2\hat\varphi_1+(1+\hat A_2)\hat\varphi_1\hat\Phi_1].\cr}
\label{4.56}
\end{equation}
Finally,
\begin{equation}
\eqal{
&\bigg|\intop_{\hat\Omega}k_{2,\tau\tau}\tilde q_{\tau\tau}dz\bigg|\le
\varepsilon\|\tilde q_{\tau\tau}\|_0^2+{c\over\varepsilon}\lambda
\|\tilde u\|_3^2+c\|\hat u\|_{2,\hat\Omega}^2\cr
&\quad+{c\over\varepsilon}(1+\|\hat p_s\|_{2,\hat\Omega}^2)
\|\hat v_s\|_{2,\hat\Omega}^2\|\hat q\|_{2,\hat\Omega}^2,\cr}
\label{4.57}
\end{equation}
where the last expression is bounded by ${c\over\varepsilon}(1+\hat A_1)\hat A_1\hat\varphi_1$.

\noindent
Employing (\ref{4.56}) and (\ref{4.57}) in (\ref{4.55}) yields
\begin{equation}
\eqal{
&{d\over dt}\intop_\Omega\bigg(\hat\varrho_s\tilde u_{\tau\tau}^2+
{1\over\hat p_s}\tilde q_{\tau\tau}^2\bigg)dz+\|\tilde u_{\tau\tau}\|_1^2+
\|\divv\tilde u_{\tau\tau}\|_0^2\cr
&\le\varepsilon\|\tilde q_{\tau\tau}\|_0^2+{c\over\varepsilon}\lambda
(\|\tilde u\|_3^2+\|\tilde q\|_2^2)+{c\over\varepsilon}
(\|\hat u\|_{2,\hat\Omega}^2+\|\hat q\|_{1,\hat\Omega}^2)\cr
&\quad+{c\over\varepsilon}[\|\hat\varrho_s\|_{2,\hat\Omega}^2
(1+\|\hat v_s\|_{2,\hat\Omega}^2)+(1+\|\hat p_s\|_{2,\hat\Omega})
\|\hat p_s\|_{3,\hat\Omega}^2+\|\hat v_s\|_{3,\hat\Omega}^2\cr
&\quad+(\|\hat p_s\|_{3,\hat\Omega}^2+\|\hat p_s\|_{2,\hat\Omega}^4)
\|\hat v_s\|_{2,\hat\Omega}^2](|\tilde u|_{2,1}^2+|\tilde q|_{2,1}^2)\cr
&\quad+{c\over\varepsilon}(1+\|\hat p_s\|_{2,\hat\Omega}^2)
(\|\hat p_s\|_{3,\hat\Omega}^2\|\hat u\|_{2,\hat\Omega}^2+
\|\hat v_s\|_{3,\hat\Omega}^2\|\hat q\|_{2,\hat\Omega}^2)\cr
&\quad+{c\over\varepsilon}(1+\|\hat p_s\|_{3,\hat\Omega}^2)
\|\hat q\|_{2,\hat\Omega}^2\|\hat u\|_{3,\hat\Omega}^2+{c\over\varepsilon}
\|\tilde{\bar f}\|_1^2.\cr}
\label{4.58}
\end{equation}
From (\ref{2.11}) we have
\begin{equation}
\eqal{
\|\tilde{\bar f}\|_1^2&\le c[\|\hat\eta\|_{2,\hat\Omega}^2\|\tilde f_s\|_1^2+
\|\hat\varrho_s\|_{2,\hat\Omega}^2\|\tilde g\|_1^2+\|\hat\eta\|_{2,\hat\Omega}^2
\|\tilde v_{st}\|_1^2\cr
&\quad+\|\hat\eta\|_{2,\hat\Omega}^2\|\tilde u_t\|_1^2+
(\|\hat v_s\|_{2,\hat\Omega}^2+\|\hat u\|_{2,\hat\Omega}^2)
(\|\hat v_s\|_{2,\hat\Omega}^2\|\hat\eta\|_{2,\hat\Omega}^2\cr
&\quad+\|\hat\eta\|_{2,\hat\Omega}^2\|\hat u\|_{2,\hat\Omega}^2+
\|\hat\varrho_s\|_{2,\hat\Omega}^2\|\hat u\|_{2,\hat\Omega}^2)]\le c
[\hat\varphi_1\|\tilde f_s\|_1^2+\hat A_1\|g\|_1^2\cr
&\quad+\hat A_1(1+\hat\varphi_1)\hat\varphi_1+(\hat A_1+\hat\varphi_1)^2
\hat\varphi_1].\cr}
\label{4.59}
\end{equation}
In view of (\ref{4.59}) and notation (\ref{2.15}) inequality (\ref{4.58})
takes the form
\begin{equation}
\eqal{
&{d\over dt}\bigg\|\sqrt{\hat\varrho_s}\tilde u_{\tau\tau},
{1\over\sqrt{\hat p_s}}\tilde q_{\tau\tau}\bigg\|_0^2+
\|\tilde u_{\tau\tau}\|_1^2+\|\divv\tilde u_{\tau\tau}\|_0^2\cr
&\le\varepsilon\|\tilde q_{\tau\tau}\|_0^2+{c\over\varepsilon}\lambda
(\|\tilde u\|_3^2+\|\tilde q\|_2^2)+{c\over\varepsilon}
(\|\hat u\|_{2,\hat\Omega}^2+\|\hat q\|_{1,\hat\Omega}^2)\cr
&\quad+{c\over\varepsilon}[(1+\hat A_1)\hat A_2\hat\varphi_1+(1+\hat A_2)
\hat\varphi_1\hat\Phi_1+\hat A_1\hat\varphi_1+\hat A_1\hat\varphi_1^2+
\hat A_1^3\hat\varphi_1\cr
&\quad+(\hat A_1+\hat\varphi_1)^2\hat\varphi_1+\hat\varphi_1\|\tilde f_s\|_1^2+
\hat A_1\|\tilde g\|_1^2].\cr}
\label{4.60}
\end{equation}
Differentiating the third component of (\ref{2.13}) with respect to $\tau$,
multiplying the result by $\tilde q_{n\tau}$ and integrating over $\hat\Omega$,
one obtains
\begin{equation}
\eqal{
&{(\mu+\nu)\over\varkappa}\intop_{\hat\Omega}\bigg(\partial_n{1\over\hat p_s}
\tilde q_t\bigg)_{,\tau}\tilde q_{n\tau}dz+{\mu+\nu\over\varkappa}
\intop_{\hat\Omega}\partial_n\bigg({1\over\hat p_s}\hat v_s\cdot\hat\nabla
\tilde q\bigg)_{,\tau}\tilde q_{n\tau}dz\cr
&\quad+\|\tilde q_{n\tau}\|_0^2\le c(\|\tilde u_{z\tau\tau}\|_0^2+
\|\tilde u_{\tau t}\|_0^2)+c(\|(\hat\varrho_s\hat v_s\cdot\hat\nabla\tilde u)_{,\tau}\|_0^2+
\|\hat\varrho_{s\tau}\tilde u_t\|_0^2\cr
&\quad+\|\tilde{\bar f}\|_1^2+\|k_{1,\tau}\|_0^2)+(\mu+\nu)\intop_{\hat\Omega}
(\tilde h_{n\tau}\tilde q_{n\tau}+k_{2n\tau}\tilde q_{n\tau})dz.\cr}
\label{4.61}
\end{equation}
Now, we examine the particular terms in (\ref{4.61}).

\noindent
We express the first two terms on the l.h.s. of (\ref{4.61}) in the form
$$\eqal{
&{\mu+\nu\over\varkappa}\intop_{\hat\Omega}\bigg({1\over\hat p_s}
\tilde q_{n\tau,t}+{1\over\hat p_s}\hat v_s\cdot\hat\nabla
\tilde q_{n\tau}\bigg)\tilde q_{n\tau}dz\cr
&\quad+{\mu+\nu\over\varkappa}\intop_{\hat\Omega}\bigg(\partial_{n\tau}^2
\bigg({1\over\hat p_s}\bigg)\tilde q_t+\partial_{n\tau}^2\bigg({1\over\hat p_s}
\hat v_s\bigg)\cdot\hat\nabla\tilde q\bigg)\tilde q_{n\tau}dz\cr
&\quad+{\mu+\nu\over\varkappa}\intop_{\hat\Omega}\bigg(\partial_z
\bigg({1\over\hat p_s}\bigg)\tilde q_{zt}+\partial_z\bigg({1\over\hat p_s}
\hat v_s\bigg)\partial_z\hat\nabla\tilde q\bigg)\tilde q_{n\tau}dz\cr
&\equiv I_1+I_2+I_3,\cr}
$$
where in $I_3$ the derivatives with respect to $n$ and $\tau$ are replaced by
derivatives with respect to $z$.

\noindent
Employing equation (\ref{2.2}) transformed to variables $z$ in $I_1$, yields
$$
I_1={\mu+\nu\over2\varkappa}{d\over dt}\intop_{\hat\Omega}{1\over\hat p_s}
\tilde q_{n\tau}^2dz-{(\mu+\nu)(\varkappa+1)\over2}\intop_{\hat\Omega}
{1\over\hat p_s}\hat{\divv}\hat v_s\tilde q_{n\tau}^2dz,
$$
where the second integral is bounded by
$$
\varepsilon\|\tilde q_{n\tau}\|_0^2+c/\varepsilon\hat A_2\|\tilde q_{n\tau}\|_0^2.
$$
Applying the H\"older and the Young inequalities we get
$$\eqal{
&|I_2|\le\varepsilon\|\tilde q_{n\tau}\|_0^2+{c\over\varepsilon}\hat A_1
(1+\hat A_1)\hat A_2\tilde\varphi_1,\cr
&|I_3|\le\varepsilon\|\tilde q_{n\tau}\|_0^2+{c\over\varepsilon}
(1+\hat A_1)\hat A_2\tilde\varphi_1.\cr}
$$
Next,
$$
\|(\hat\varrho_s\hat v_s\hat\nabla\tilde u)_{,\tau}\|_0^2\le c
\|\hat\varrho_s\|_{2,\hat\Omega}^2\|\hat v_s\|_{2,\hat\Omega}^2
\|\tilde u\|_2^2\le c\hat A_1^2\tilde\varphi_1,
$$
$\|\tilde{\bar f}\|_1^2$ is bounded by (\ref{4.59}),
$$
\|k_{1,\tau}\|_0^2\le c\lambda(\|\tilde u_{zzz}\|_0^2+\|\tilde q_{zz}\|_0^2)+c
(\|\hat u\|_{2,\hat\Omega}^2+\|\hat q\|_{1,\hat\Omega}^2).
$$
Similarly, as in the case of estimates (\ref{4.56}) and (\ref{4.57}) we have
$$
\bigg|\intop_{\hat\Omega}\tilde h_{n\tau}\tilde q_{n\tau}dz\bigg|\le\varepsilon
\|\tilde q_{n\tau}\|_0^2+c/\varepsilon[(1+\hat A_1)\hat A_2\hat\varphi_1+
(1+\hat A_2)\hat\varphi_1\hat\Phi_1]
$$
and
$$\eqal{
&\bigg|\intop_{\hat\Omega}k_{2,n\tau}\tilde q_{n\tau}dz\bigg|\le\varepsilon
\|\tilde q_{n\tau}\|_0^2+{c\over\varepsilon}\lambda\|\tilde u\|_3^2+c
\|\hat u\|_{2,\hat\Omega}^2+{c\over\varepsilon}(1+\hat A_1)\hat A_1\hat\varphi_1.\cr}
$$
Employing the above estimates in (\ref{4.61}), one gets
\begin{equation}
\eqal{
&{d\over dt}\intop_{\hat\Omega}{1\over\hat p_s}\tilde q_{n\tau}^2dz+
\|\tilde q_{n\tau}\|_0^2\le c(\|\tilde u_{z\tau\tau}\|_0^2+
\|\tilde u_{\tau t}\|_0^2)\cr
&\quad+c\lambda(\|\tilde u_{zzz}\|_0^2+\|\tilde q_{zz}\|_0^2)+c
(\|\hat u\|_{2,\hat\Omega}^2+\|\hat q\|_{1,\hat\Omega}^2)+
c\hat\varphi_1\|\tilde f_s\|_1^2\cr
&\quad+c\hat A_1\|\tilde g\|_1^2+c[(1+\hat A_1)^2\hat A_2\hat\varphi_1+\hat A_1^2\tilde\varphi_1+\hat A_1
\hat\varphi_1^2+(1+\hat A_2)\hat\varphi_1\hat\Phi_1].\cr}
\label{4.62}
\end{equation}
Differentiating the third component of (\ref{2.14}) with respect to $\tau$ and
taking the $L_2$-norm yields
\begin{equation}
\eqal{
&\|(\divv\tilde u)_{,n\tau}\|_0^2\le c(\|\tilde u_{z\tau\tau}\|_0^2+
\|\tilde u_{\tau t}\|_0^2+\|\tilde q_{n\tau}\|_0^2)+c\|\hat\varrho_s\|_2^2
\|\tilde u_t\|_1^2\cr
&\quad+c\lambda(\|\tilde u_{zzz}\|_0^2+\|\tilde q_{zz}\|_0^2)+c
(\|\hat u\|_{2,\hat\Omega}^2+\|\hat q\|_{1,\hat\Omega}^2)\cr
&\quad+c[\hat\varphi_1\|\tilde f_s\|_1^2+\hat A_1\|g\|_1^2+
(\hat A_1+\hat\varphi_1)\hat\varphi_1+(\hat A_1+\hat\varphi_1)^2\hat\varphi_1].\cr}
\label{4.63}
\end{equation}
Consider now the following Stokes problem
\begin{equation}
\eqal{
&-\mu\Delta\tilde u+\nabla\tilde q=-\hat\varrho_s\tilde u_t+\nu\nabla\divv
\tilde u+\tilde{\bar f}+k_1-\hat\varrho_s\hat v_s\cdot\hat\nabla\tilde u\cr
&\divv\tilde u=\divv\tilde u\quad {\rm in}\ \ \hat\Omega,\cr
&\tilde u|_{\hat S}=0.\cr}
\label{4.64}
\end{equation}
Differentiating (\ref{4.64}) with respect to $\tau$, we obtain
\begin{equation}
\eqal{
&\|\tilde u_\tau\|_2^2+\|\tilde q_\tau\|_1^2\le c(\|\tilde u_{\tau t}\|_0^2+
\|\divv\tilde u_\tau\|_1^2+\|\tilde{\bar f}\|_1^2\cr
&\quad+\|\hat\varrho_{s\tau}\tilde u_t\|_0^2+\|k_{1,\tau}\|_0^2+
\|(\hat\varrho_s\hat v_s\hat\nabla\tilde u)_{,\tau}\|_0^2)\cr
&\le c(\|\tilde u_{\tau t}\|_0^2+\|\divv\tilde u_\tau\|_1^2)+c\lambda
(\|\tilde u_{zzz}\|_0^2+\|\tilde q_{zz}\|_0^2)\cr
&\quad+c(\|\hat u\|_{2,\hat\Omega}^2+\|\hat q\|_{1,\hat\Omega}^2)+c\hat\varphi_1
\|\tilde f_s\|_1^2+c\hat A_1\|\tilde g\|_1^2\cr
&\quad+c\hat A_1(1+\hat A_1)\hat\varphi_1+\hat A_1\hat\varphi_1^2
+c\hat{\varphi}_1^3.\cr}
\label{4.65}
\end{equation}
Adding appropriately (\ref{4.60}), (\ref{4.62}), (\ref{4.63}) and (\ref{4.65})
we have
\begin{equation}
\eqal{
&{d\over dt}\bigg\|\sqrt{\hat\varrho_s}\tilde u_{\tau\tau},
{1\over\sqrt{\hat p_s}}\tilde q_{\tau\tau},{1\over\sqrt{\hat p_s}}
\tilde q_{n\tau}\|_0^2+\|\tilde u_\tau\|_2^2+\|\tilde q_\tau\|_1^2\cr
&\le c\|\tilde u_t\|_1^2+c\lambda(\|\tilde u_{zzz}\|_0^2+\|\tilde q_{zz}\|_0^2)
+c(\|\hat u\|_{2,\hat\Omega}+\|\hat q\|_{1,\hat\Omega}^2)\cr
&\quad+c\hat\varphi_1\|\tilde f_s\|_1^2+c\hat A_1\|\tilde g\|_1^2+c
[(1+\hat A_1)\hat A_2\hat\varphi_1+(1+\hat A_2)\hat\varphi_1\hat\Phi_1\cr
&\quad+(\hat A_1+\hat\varphi_1)^2\hat\varphi_1+\hat A_1(1+\hat\varphi_1)
\hat\varphi_1]\cr
&\equiv c\|\tilde u_t\|_1^2+c\lambda(\|\tilde u_{zzz}\|_0^2+
\|\tilde q_{zz}\|_0^2)+c(\|\hat u\|_{2,\hat\Omega}^2+
\|\hat q\|_{1,\hat\Omega}^2)\cr
&\quad+c\hat X_4^2.\cr}
\label{4.66}
\end{equation}
Differentiating the third component of (\ref{2.13}) with respect to $n$,
multiplying the result by $\tilde q_{nn}$ and integrating over $\hat\Omega$,
one derives
\begin{equation}
\eqal{
&{d\over dt}\bigg\|{1\over\sqrt{\hat p_s}}\tilde q_{nn}\|_0^2+
\|\tilde q_{nn}\|_0^2\le c(\|\tilde u_t\|_1^2+\|\tilde u_{zz\tau}\|_0^2)\cr
&\quad+c\lambda(\|\tilde u_{zzz}\|_0^2+\|\tilde q_{zz}\|_0^2)+c
(\|\hat u\|_{2,\hat\Omega}^2+\|\hat q\|_{1,\hat\Omega}^2)\cr
&\quad+c\hat X_4^2.\cr}
\label{4.67}
\end{equation}
The $L_2$-norm of the derivative with respect to $n$ of the third component
of (\ref{2.14}) is bounded by
\begin{equation}
\eqal{
&\|(\divv\tilde u)_{,nn}\|_0^2\le c(\|\tilde u_{zz\tau}\|_0^2+
\|\tilde q_{nn}\|_0^2+\|\tilde u_t\|_1^2)\cr
&\quad+c\lambda(\|\tilde u_{zzz}\|_0^2+\|\tilde q_{zz}\|_0^2)+c
(\|\hat u\|_{2,\hat\Omega}^2+\|\hat q\|_{1,\hat\Omega}^2)\cr
&\quad+c\hat X_4^2.\cr}
\label{4.68}
\end{equation}
For solution to problem (\ref{4.64}) we have
\begin{equation}
\eqal{
&\|\tilde u\|_3^2+\|\tilde q\|_2^2\le c(\|\nabla\divv\tilde u\|_1^2+
\|\tilde u_t\|_1^2)\cr
&\quad+c\lambda(\|\tilde u_{zzz}\|_0^2+\|\tilde q_{zz}\|_0^2)+c
(\|\hat u\|_{2,\hat\Omega}^2+\|\hat q\|_{1,\hat\Omega}^2)\cr
&\quad+c\hat X_4^2.\cr}
\label{4.69}
\end{equation}
Adding appropriately (\ref{4.66})--(\ref{4.69}) gives
\begin{equation}
\eqal{
&{d\over dt}\bigg\|\sqrt{\hat\varrho_s}\tilde u_{\tau\tau},
{1\over\sqrt{\hat p_s}}\tilde q_{zz}\bigg\|_0^2+\|\tilde u\|_3^2+
\|\tilde q\|_2^2\cr
&\le c\|\tilde u_t\|_1^2+c\lambda(\|\tilde u_{zzz}\|_0^2+\|\tilde q_{zz}\|_0^2)
+c(\|\hat u\|_{2,\hat\Omega}^2+\|\hat q\|_{1,\hat\Omega}^2)\cr
&\quad+c\hat X_4^2.\cr}
\label{4.70}
\end{equation}
Passing to the old variables $x$ in (\ref{4.70}), deriving an inequality
similar to (\ref{4.70}) in an interior subdomain, summing the inequalities
over all neighborhoods of the partition of unity and assuming that $\lambda$
is sufficiently small, we obtain
\begin{equation}
\eqal{
&{d\over dt}\intop_\Omega\bigg(\varrho_s\bar u_{\tau\tau}^2+{1\over p_s}
q_{xx}\bigg)dx+\|u\|_3^2+\|q\|_2^2\le c(\|u_t\|_1^2\cr
&\quad+\|u\|_2^2+\|q\|_1^2)+cX_4^2,\cr}
\label{4.71}
\end{equation}
where $\bar u_{\tau\tau}$ means that in a~neighborhood of the boundary there
are only tangent derivatives.

\noindent
To derive (\ref{4.52}) from (\ref{4.71}) we need the expression
\begin{equation}
\eqal{
&{d\over dt}\intop_\Omega\varrho_su_{xx}^2dx=2\intop_\Omega\varrho_su_{xxt}u_{xx}dx+\intop_\omega\varrho_{st}u_{xx}^2dx\cr
&\le\varepsilon(\|u_{xxt}\|_0^2+\|u_{xx}\|_1^2)+c/\varepsilon
(\|\varrho_s\|_2^2+\|\varrho_{st}\|_1^2)\|u_{xx}\|_0^2.\cr}
\label{4.72}
\end{equation}
Using (\ref{4.72}) in (\ref{4.71}) and choosing $\varepsilon$ sufficiently
small implies (\ref{4.52}). This concludes the proof.
\end{proof}

\begin{lemma}\label{l4.5}
For sufficiently regular solutions we have
\begin{equation}\eqal{
&{d\over dt}\bigg\|\sqrt{\varrho_s}u,\sqrt{\varrho_s}u_x,\sqrt{\varrho_s}
u_{xx},\sqrt{\varrho}u_t,\sqrt{\varrho}u_{xt},{1\over\sqrt{p_s}}q,
{1\over\sqrt{p_s}}q_t,{1\over\sqrt{p_s}}q_x,{1\over\sqrt{p_s}}q_{xx}\bigg\|_0^2\cr
&\quad+|u|_{3,2}^2+|q|_{2,1}^2\le c(X_0^2+X_3^2+X_4^2),\cr}
\label{4.73}
\end{equation}
where $X_0$ is introduced in (\ref{4.1}), $X_3$ in (\ref{4.46}) and $X_4$ in
(\ref{4.53}).
\end{lemma}

\begin{proof}
Applying some interpolation inequalities we obtain from (\ref{4.52}) the
inequality
\begin{equation}
\eqal{
&{d\over dt}\bigg\|\sqrt{\varrho_s}u_{xx},{1\over\sqrt{p_s}}q_{xx}\bigg\|_0^2+
\|u\|_3^2+\|q\|_2^2\le\varepsilon\|u_{xxt}\|_0^2\cr
&\quad+c\|u,u_t,q\|_0^2+cX_4^2.\cr}
\label{4.74}
\end{equation}
Adding appropriately (\ref{4.1}) and (\ref{4.74}) yields
\begin{equation}
\eqal{
&{d\over dt}\bigg\|\sqrt{\varrho_s}u,\sqrt{\varrho}u_t,\sqrt{\varrho_s}u_{xx},
{1\over\sqrt{p_s}}q,{1\over\sqrt{p_s}}q_t,{1\over\sqrt{p_s}}q_{xx}\bigg\|_0^2\cr
&\quad+\|u\|_3^2+\|u_t\|_1^2+\|q\|_2^2+\|q_t\|_0^2\le\varepsilon
\|u_{xxt}\|_0^2+c(X_0^2+X_4^2).\cr}
\label{4.75}
\end{equation}
Adding (\ref{4.45}) and (\ref{4.75}) and using that $\varepsilon$ is
sufficiently small we get (\ref{4.73}). This concludes the proof.
\end{proof}

\noindent
Since in (\ref{4.73}) the norm $\big\|{1\over\sqrt{p_s}}q_{xt}\big\|_0$ does
not appear under the time derivative we need

\begin{lemma}\label{l4.6}
For sufficiently regular solutions we have
\begin{equation}
\eqal{
&{d\over dt}\bigg\|{1\over\sqrt{p_s}}q_{xt}^2\bigg\|_0^2\le\varepsilon
\|q_{xt}\|_0^2+c/\varepsilon\|u_{xxt}\|_0^2+c/\varepsilon[[(1+A_1)A_1\cr
&\quad+A_1^2(1+A_1)]\varphi_1+(1+A_1)\varphi_1\Phi_1+
[1+A_4(1+A_1)]\varphi_1^2].\cr}
\label{4.76}
\end{equation}
\end{lemma}

\begin{proof}
Differentiate (\ref{1.10}) with respect to $x$ and $t$, multiply the result
by $q_{xt}$ and integrate over $\Omega$. Then we obtain
\begin{equation}
\eqal{
&\intop_\Omega\bigg({1\over p_s}(q_t+v_s\cdot\nabla q)\bigg)_{xt}q_{xt}dx+
\varkappa\intop_\Omega\divv u_{xt}q_{xt}dx\cr
&=-\intop_\Omega\bigg({4\over p_s}\nabla(p_s+q)\bigg)_{xt}q_{xt}dx-\varkappa
\intop_\Omega\bigg({q\over p_s}\divv(v_s+u)\bigg)_{xt}q_{xt}dx\cr
&\equiv J+K.\cr}
\label{4.77}
\end{equation}
The first term on the l.h.s. equals
\begin{equation}
\eqal{
&\intop_\Omega\bigg({1\over p_s}q_{xtt}q_{xt}+{v_s\over p_s}\cdot\nabla q_{xt}q_{xt}\bigg)dx+\intop_\Omega\bigg[\bigg({1\over p_s}\bigg)_{,x}q_{tt}q_{xt}+
\bigg({v_s\over p_s}\bigg)_{,x}q_{xt}^2\bigg]dx\cr
&\quad+\intop_\Omega\bigg[\bigg({1\over p_s}\bigg)_{,t}q_{xt}^2+
\bigg({v_s\over p_s}\bigg)_{,t}q_{xx}q_{xt}\bigg]dx+\intop_\Omega\bigg[
\bigg({1\over p_s}\bigg)_{xt}q_t+\bigg({v_s\over p_s}\bigg)_{,xt}q_x\bigg]
q_{xt}dx\cr
&\equiv\sum_{i=1}^4I_i.\cr}
\label{4.78}
\end{equation}
Now we examine the terms in (\ref{4.78}).
$$
I_1={1\over2}\intop_\Omega\bigg[{1\over p_s}\partial_tq_{xt}^2+{v_s\over p_s}
\cdot\nabla q_{xt}^2\bigg]dx
$$
From (\ref{2.2}) we have
\begin{equation}
{1\over2}\intop_\Omega\bigg[\bigg({1\over p_s}\bigg)_{,t}+\divv
\bigg({v_s\over p_s}\bigg)\bigg]q_{xt}^2dx-(\varkappa+1)\intop_\Omega
{1\over p_s}\divv v_sq_{xt}^2dx=0.
\label{4.79}
\end{equation}
Adding (\ref{4.79}) to $I_1$ and using that $v_s\cdot\bar n|_S=0$ we obtain
$$
I_1={1\over2}{d\over dt}\intop_\Omega{1\over p_s}q_{xt}^2dx-(\varkappa+1)
\intop_\Omega{1\over p_s}\divv v_sq_{xt}^2dx,
$$
where the second integral is bounded by
$$
\varepsilon\|q_{xt}\|_0^2+c/\varepsilon\|v_s\|_3^2\|q_{xt}\|_0^2.
$$
Next
$$\eqal{
|I_2|&\le c\intop_\Omega[|p_{sx}|\,|q_{tt}|\,|q_{xt}|+(|v_{sx}|+|v_s|\,
|p_{sx}|)q_{xt}^2]dx\cr
&\le\varepsilon\|q_{xt}\|_0^2+c/\varepsilon[\|p_s\|_3^2\|q_{tt}\|_0^2+
(\|v_s\|_3^2+\|v_s\|_2^2\|p_s\|_3^2)\|q_{xt}\|_0^2].\cr}
$$
From (\ref{1.10}) we calculate
\begin{equation}
\eqal{
\|q_{tt}\|_0^2&\le\|(v_s\cdot\nabla q)_t\|_0^2+\|(p_s\divv u)_t\|_0^2+
\|(u\cdot\nabla(p_s+q))_t\|_0^2\cr
&\quad+\|(q\divv(v_s+u))_t\|_0^2\le c(A_1+\varphi_1)\varphi_1,\cr}
\label{4.80}
\end{equation}
where we used (\ref{2.15}). Continuing, we have
$$
|I_2|\le\varepsilon\|q_{xt}\|_0^2+c/\varepsilon[A_2(A_1+\varphi_1)\varphi_1+(1+A_1)A_2\varphi_1].
$$
Next, we examine
$$\eqal{
|I_3|&\le c\intop_\Omega[|p_{st}|q_{xt}^2+(|v_{xt}+|v_s|\,|p_{st})|q_{xx}|\,|q_{xt}|]dx\cr
&\le\varepsilon\|q_{xt}\|_0^2+c/\varepsilon(1+A_1)A_3\varphi_1.\cr}
$$
Finally, we calculate
$$\eqal{
|I_4|&\le c\intop_\Omega[(|p_{st}|\,|p_{sx}|+|p_{sxt}|)|q_t|+
(|v_{sxt}|+|v_{sx}|\,|p_{st}|\cr
&\quad+|v_{st}|\,|p_{sx}|+|v_s|\,|p_{st}|\,|p_{sx}|)|q_x|]|q_{xt}|dx\cr
&\le\varepsilon\|q_{xt}\|_0^2+c/\varepsilon[A_1^2+A_1^3+A_3]\varphi_1.\cr}
$$
The second term on the l.h.s. of (\ref{4.77}) is bounded by
$$
c(\|q_{xt}\|_0^2+\|u_{xxt}\|_0^2).
$$
The first term on the r.h.s. of (\ref{4.77}) equals
$$\eqal{
&-\intop_\Omega\bigg({u\over p_s}\cdot\nabla p_s\bigg)_{xt}q_{xt}dx-
\intop_\Omega{u\over p_s}\nabla q_{xt}q_{xt}dx-\intop_\Omega
\bigg({u\over p_s}\bigg)_x\nabla q_{xt}q_{xt}dx\cr
&\quad-\intop_\Omega\bigg({u\over p_s}\bigg)_t\nabla q_xq_{xt}dx-\intop_\Omega
\bigg({u\over p_s}\bigg)_{xt}\cdot\nabla qq_{xt}dx\equiv\sum_{i=1}^5J_i.\cr}
$$
Continuing, we have
$$
|J_1|\le\varepsilon\|q_{xt}\|_0^2+c/\varepsilon[A_1^2(1+A_1)\varphi_1+
(1+A_1)A_3\varphi_1+A_1\Phi_1].
$$
$$
J_2=-{1\over2}\intop_\Omega{u\over p_s}\cdot\nabla q_{xt}^2dx={1\over2}
\intop_\Omega\nabla\bigg({u\over p_s}\bigg)q_{xt}^2dx.
$$
Hence,
$$
|J_2|\le\varepsilon\|q_{xt}\|_0^2+c/\varepsilon(\Phi_1+A_2\varphi_1)\varphi_1.
$$
$J_3$ takes the form
$$
J_3=-\intop_\Omega\bigg({u_x\over p_s}-{u\over p_s^2}p_{sx}\bigg)q_{xt}^2dx
$$
so
$$
|J_3|\le\varepsilon\|q_{xt}\|_0^2+c/\varepsilon(\Phi_1+A_2\varphi_1)\varphi_1
$$
Next
$$
J_4=-\intop_\Omega\bigg({u_t\over p_s}-{u\over p_s^2}p_{st}\bigg)q_{xx}q_{xt}dx.
$$
Hence
$$
|J_4|\le\varepsilon\|q_{xt}\|_0^2+c/\varepsilon(\Phi_1+A_3\varphi_1)\varphi_1.
$$
Finally
$$
|J_5|\le\varepsilon\|q_{xt}\|_0^2+c/\varepsilon[\Phi_1+(A_1+A_1^2+A_3)
\varphi_1]\varphi_1.
$$
Next
$$
|K|\le\varepsilon\|q_{xt}\|_0^2+c/\varepsilon[(1+A_1)^2(A_1+\varphi_1)\varphi_1
+(1+A_1)(A_2+A_3+\Phi_1)\varphi_1].
$$
Employing the above estimates in (\ref{4.77}) implies (\ref{4.76}). This
concludes the proof.
\end{proof}

\noindent
Let
\begin{equation}
\eqal{
\varphi_0(t)&=\bigg\|\sqrt{\varrho_s}u,\sqrt{\varrho_s}u_x,\sqrt{\varrho_s}
u_{xx},\sqrt{\varrho}u_t,\sqrt{\varrho}u_{xt},{1\over\sqrt{p_s}}q,
{1\over\sqrt{p_s}}q_t,{1\over\sqrt{p_s}}q_{xt},\cr
&\quad{1\over\sqrt{p_s}}q_x,{1\over\sqrt{p_s}}q_{xx}\bigg\|_{L_2}^2.\cr}
\label{4.81}
\end{equation}
Since $\varrho_*\le\varrho_s\le\varrho^*$,
$A\varrho_*^\varkappa\le p_s\le A(\varrho^*)^\varkappa$ the local solution
satisfies also that
$$
{1\over2}\varrho_*\le\varrho\le2\varrho^*.
$$
Hence, there exist constants $c_1$, $c_2$, $c_1<c_2$ such that
\begin{equation}
c_1\varphi_1\le\varphi_0\le c_2\varphi_1.
\label{4.82}
\end{equation}
Then the above lemmas imply

\begin{lemma}\label{l4.7}
For sufficiently regular solutions we have
\begin{equation}
\eqal{
&{d\over dt}\varphi_0+\Phi_1\le c(|f_s|_{1,0}^2+A_4(1+A_1)+A_1+A_1^2+A_1^3+
A_1^4)\varphi_0\cr
&\quad+c(A_1\varphi_0+(1+\varphi_0)\varphi_0+\varphi_0^4)\varphi_0+c
(1+A_4)\varphi_1\Phi_1+cA_1|g|_{1,0}^2\cr
&\equiv c_0B_1\varphi_0+c_0B_2(\varphi_0+\varphi_0^2+\varphi_0^4)\varphi_0+
c_0B_3\varphi_1\Phi_1+c_0G.\cr}
\label{4.83}
\end{equation}
\end{lemma}

\noindent
The differential inequality (\ref{4.83}) is proved for the local solutions.
Our aim is to extend the local existence step by step in time.

\begin{lemma}\label{l4.8}
Let $B_i$, $i=1,2,3$, and $G$ be defined in (\ref{4.83}). Assume that
\begin{enumerate}
\item $\bar B_i=\sup_{k\in\N_0}\sup_{t\in[kT,(k+1)T]}B_i(t)$, $i=2,3$, $\sup_k\intop_{kT}^{(k+1)T}B_1(t)dt<\infty$.
\item $c_0\bar B_3\varphi_1\le{1\over2}$.
\item $\gamma_*$ is a~constant so small that
$$
{1\over2}-c_0\bar B_2(1+\gamma_*+\gamma_*^3)\gamma_*^2\ge{c_*\over 2},\quad 0<c_*<1.
$$
\item Let for $\gamma\le\gamma_*$,
$$
\varphi(0)\le\gamma,\quad c_0G(t)\le{c_*\over4}\gamma,\quad t\in\R_+.
$$
\end{enumerate}
Then
\begin{equation}
\varphi(t)\le\gamma\quad \textsl{for}\ \ t\in\R_+.
\label{4.84}
\end{equation}
\end{lemma}

\begin{proof}
In view of Assumption 2 and that $\varphi_0\le\Phi_1$ we obtain from
(\ref{4.83}) the inequality
\begin{equation}
{d\over dt}\varphi_0\le-\bigg({1\over2}-c_0\bar B_2(1+\varphi_0+\varphi_0^3)
\varphi_0\bigg)\varphi_0+c_0G+c_0B_1\varphi_0.
\label{4.85}
\end{equation}
To prove the lemma we examine inequality (\ref{4.85}) step by step in time.
Therefore, we examine it in the time interval $[kT,(k+1)T]$, $T>0$ is given
and $k\in\N_0$. Then we assume
\begin{equation}
\varphi_0(kT)\le\gamma,\quad c_0G(t)\le{c_*\over4}\gamma\quad t\in[kT,(k+1)T].
\label{4.86}
\end{equation}
Our aim is to show that
\begin{equation}
\varphi_0((k+1)T)\le\gamma.
\label{4.87}
\end{equation}
Let us introduce the quantity
\begin{equation}
\eta(t)=\exp\bigg(-c\intop_{kT}^tB_1(t')dt'\bigg)\varphi_0(t)
\label{4.88}
\end{equation}
At $t=kT$ we have
\begin{equation}
\eta(kT)=\varphi_0(kT)\le\gamma.
\label{4.89}
\end{equation}
Introducing the quantity
\begin{equation}
G_0(t)=G(t)\exp\bigg(-c_0\intop_{kT}^tB_1(t')dt'\bigg)
\label{4.90}
\end{equation}
we express (\ref{4.85}) in the form
\begin{equation}
{d\over dt}\eta\le-(1/2-c_0\bar B_2(1+\varphi_0+\varphi_0^3)\varphi_0)\eta+c_0G_0
\label{4.91}
\end{equation}
Suppose that
$$\eqal{
t_*&=\inf\{t\in(kT,(k+1)T]:\ \varphi_0(t)>\gamma\}\cr
&=\inf\bigg\{t\in(kT,(k+1)T]:\ \eta(t)>\gamma\exp\bigg(-c_0\intop_{kT}^tB_1(t')
dt'\bigg)\bigg\}>kT.\cr}
$$
By Assumption 3 inequality (\ref{4.91}) takes the form
\begin{equation}
{d\over dt}\eta\le-{c_*\over2}\eta+G_0.
\label{4.92}
\end{equation}
Clearly, we have
\begin{equation}
\eta(t_*)=\gamma\exp\bigg(-c_0\intop_{kT}^{t_*}B_1(t)dt\bigg)
\label{4.93}
\end{equation}
and
\begin{equation}
\eta(t)>\gamma\exp\bigg(-c_0\intop_{kT}^tB_1(t')dt'\bigg)\quad {\rm for}\ \
t>t_*.
\label{4.94}
\end{equation}
By Assumption 4 we obtain from (\ref{4.92}) the relation
\begin{equation}
{d\over dt}\eta|_{t=t_*}\le c_*\bigg(-{\gamma\over2}+{\gamma\over4}\bigg)\exp
\bigg(-c_0\intop_{kT}^{t_*}B_1(t)dt\bigg)<0,
\label{4.95}
\end{equation}
so it contradicts to (\ref{4.93}) and (\ref{4.94}). Hence
$$
\eta(t)<\gamma\exp\bigg(-\intop_{kT}^tB_1(t')dt'\bigg)
$$
and
$$
\varphi_0(t)<\gamma.
$$
Hence (\ref{4.85}) holds. By the induction we prove the lemma for all
$t\in\R_+$. This concludes the proof.
\end{proof}

\noindent
Let us denote
$$\eqal{
X_s(kT,(k+1)T)&=\|\varrho_s,v_s\|_{L_\infty(kT,(k+1)T;\Gamma_2^3(\Omega))}^2\cr
&\quad+\|v_{x,tt}\|_{L_\infty(kT,(k+1)T;L_2(\Omega))}^2.\cr}
$$

\begin{remark}\label{r4.8}
Let the asusmptions of Lemma \ref{l4.8} hold. Integrating (\ref{4.83}) with
respect to time from $kT$ to $(k+1)T$ yields
$$\eqal{
&\sup_k\|u,q\|_{{\frak M}(\Omega\times(kT,(k+1)T))}\le[\varphi(\sup_kX_s
(kT,(k+1)T))\gamma\cr
&\quad+\sup_k\|f_s\|_{L_\infty(kT,(k+1)T;\Gamma_0^1(\Omega))}^2\gamma+\sup_k
\|g\|_{L_\infty(kT,(k+1)T;\Gamma_0^1(\Omega))}]\cr
&\le D\gamma,\cr}
$$
where $D$ is a~constant.
\end{remark}

\bibliographystyle{amsplain}

\end{document}